\documentclass[11pt,reqno]{article}
\usepackage[top = 2cm, bottom = 2cm, left = 2cm, right = 2cm]{geometry}
 \usepackage[usenames,dvipsnames]{color}

\usepackage{theoremref}
\usepackage{enumerate,enumitem}

    \newcommand{\lab}[1]{\label{#1}}              
     \newcommand{\thlab}[1]{\thlabel{#1}} 
 \usepackage{amssymb}
 \usepackage{amsfonts,mathrsfs}
\usepackage{amsmath,bbm,amsthm}
  \usepackage{graphicx, mathabx}

\newcommand{\remove}[1]{}
\newcommand{\sign}{\ensuremath{\mathrm{sign}}} 

\newcommand{\bee}{\begin{equation}}
\newcommand{\ee}{\end{equation}}
\newcommand\eqn[1]{(\ref{#1})}
\newcommand{\bel}[1]{\bee\lab{#1}}

\newcommand\se{\subseteq}
\newcommand\sm{\setminus}
\newcommand{\bad}{\ensuremath {{\bf bad}}}
\newcommand{\Bad}{\ensuremath {B}}
\newcommand{\savg}{s}
\newcommand{\tavg}{t}
\newcommand{\DeltaA}{\Delta_S}
\newcommand{\DeltaB}{\Delta_T}
\newcommand{\hatDeltaA}{\hat\Delta_S}
\newcommand{\hatDeltaB}{\hat\Delta_T}
\newcommand{\dvA}{\sv}
\newcommand{\dvB}{\tv}
\newcommand{\dcA}{s}
\newcommand{\dcB}{t}
\newcommand{\corr}[1]{\ensuremath{\correct (#1)}} 	

\newcommand{\degspread}{\varphi}

\def\Gnp{{\cal G}(n,p)}
\def\eps{\varepsilon}
\newcommand{\Gnm}{\mathcal{G}(n,m)}
\def\Num{\mathcal{N}}
\def\Yc{{\cal Y}}
\def\yv{{\bf y}}

\def\Pgr{{P^*}}
\def\Ygr{{Y^*}}
\def\apx{{H}}
\def\correct{{\tilde H}}

\def\cA{{\cal A}}			
\newcommand{\cB}{\mathcal{B}}
\def\Cc{{\cal C}}
\def\cC{{\cal C}}
\def\cD{{\cal D}}		
\newcommand{\G}{\mathcal{G}}
\newcommand{\cG}{\mathcal{G}}
\def\cN{{\cal N}}
\def\oA{{\vec{\cA}}}		
\def\Pc{{\cal P}}
\def\Rc{{\cal R}}
\newcommand{\cS}{\mathcal{S}}

\def\D{{\mathfrak{D}}}	
\def\R{\mathbbm{R}}		
\def\W{{\mathfrak{W}}} 	
\def\Z{\mathbbm{Z}_{\ge 0}}	

\def\pr{{\bf P}}			
\def\Var{{\bf Var}}
\def\ex{{\bf E}}
\newcommand{\Bin}{\ensuremath{\mathrm{Bin}}}

\def\dv{{\bf d}}
\def\sv{{\bf s}}
\def\tv{{\bf t}}
\def\dw{{\bf d}'}
\def\pv{{\bf p}}
\def\rv{{\bf r}}

\def\ve{{\bf e}}
\def\eb{{\ve_b}}
\def\ea{{\ve_a}}
\def\ev{{\ve_v}}
\def\ew{{\ve_w}}

\newcommand{\bi}{\mathrm{bi}}
\newcommand{\di}{\mathrm{di}}

\newcommand{\mate}{\mathrm{mate}}
\renewcommand{\mate}[1]{\ensuremath{#1'}}

\usepackage{mathtools}
 
 \makeatletter
\DeclareRobustCommand\widecheck[1]{{\mathpalette\@widecheck{#1}}}
\def\@widecheck#1#2{%
    \setbox\z@\hbox{\m@th$#1#2$}%
    \setbox\tw@\hbox{\m@th$#1%
       \widehat{%
          \vrule\@width\z@\@height\ht\z@
          \vrule\@height\z@\@width\wd\z@}$}%
    \dp\tw@-\ht\z@
    \@tempdima\ht\z@ \advance\@tempdima2\ht\tw@ \divide\@tempdima\thr@@
    \setbox\tw@\hbox{%
       \raise\@tempdima\hbox{\scalebox{1}[-1]{\lower\@tempdima\box
\tw@}}}%
    {\ooalign{\box\tw@ \cr \box\z@}}}
\makeatother

\newtheorem{thm}{Theorem}[section]

\newtheorem{lemma}[thm]{Lemma}
\newtheorem{proposition}[thm]{Proposition}

\newtheorem{claim}[thm]{Claim}
\newtheorem{remark}[thm]{Remark}
\newtheorem{definition}[thm]{Definition}

\date{}

\begin{document}
\title{Asymptotic enumeration of digraphs and bipartite graphs  by degree sequence\thanks{Research supported by ARC Discovery Project DP180103684.}}

 \author{Anita Liebenau\thanks{Supported by an ARC DECRA Fellowship grant DE170100789.}\\
 {\small School of Mathematics and Statistics}\\{\small UNSW Sydney NSW 2052}\\
 {\small Australia} \\
 {\small  \tt{a.liebenau@unsw.edu.au}}
 \and Nick Wormald\thanks{Supported by an ARC Australian Laureate Fellowship grant FL120100125.}
 \\
 {\small School of Mathematical Sciences}\\{\small Monash University VIC 3800}\\
 {\small Australia} \\
{\small  \tt{ nick.wormald@monash.edu}}
 }

 \maketitle

 \begin{abstract}
We provide asymptotic formulae for the numbers of bipartite graphs  with given degree sequence, and of loopless digraphs with given in- and out-degree sequences, for a wide range of parameters. Our results cover medium range densities and close the gaps between the results known for the sparse and dense ranges. In the case of bipartite graphs, these results were proved by Greenhill, McKay and Wang in 2006 and by Canfield, Greenhill and McKay in 2008, respectively. Our method also essentially covers the sparse range, for which much less was known in the case of loopless digraphs.  For the range of densities which our results cover, they imply that the degree sequence of a random bipartite graph with $m$ edges is accurately modelled by a sequence of independent binomial random variables, conditional upon the sum of variables in each part being equal to $m$. A similar model also holds for loopless digraphs.   

 \end{abstract}



\section{Introduction} 
Enumeration of discrete structures with local constraints has attracted the interest of many researchers and has applications in various areas such as coding theory, statistics and neurostatistical analysis. Exact formulae are often hard to derive or infeasible to compute. Asymptotic formulae are therefore sought and often provide sufficient information for the aforementioned applications. In this paper we find such formulae for bipartite graphs  with given degree sequence, or loopless digraphs with given in- and out-degree sequences. Our results imply that the degree sequence of a random digraph or bipartite graph with $m$ edges is close to a sequence of independent binomial random variables, conditional upon the sum of degrees in each part being equal to $m$. 

We frame all our arguments in terms of bipartite graphs: as noted below, digraphs are equivalent to ``balanced" bipartite graphs. Thus, if loops are not forbidden, the digraph enumeration problem is the same as the bipartite one.  
The loopless case for digraphs is equivalent to bipartite graphs with a forbidden perfect matching. 
Our results on counting bipartite graphs with a given degree sequence imply equivalent results on counting $0$-$1$ matrices with given row and column sums. Similarly, counting (loopless) digraphs is equivalent to counting square $0$-$1$ matrices with given row and column sums where the entries on the diagonal are required to be 0.

Our results are obtained via the method of degree switchings and contraction mappings recently introduced by the authors in~\cite{lw2018} to count the number of  ``nearly" regular  graphs of a given degree sequence for  medium-range densities, and a wider range of degree sequences for low densities.  The basic structure of the argument is very similar in the present case, but it needs significant modifications to account for the fact that we are dealing with bipartite graphs and certain edges are not allowed. 
\subsection{Enumeration results}
 The formulae in~\cite{lw2018} are stated in terms of a relationship between the degree sequence of the Erd\H{o}s-R\'enyi random graph and a sequence of  independent binomial random variables. We shall do the same here for appropriate bipartite random graphs and suitable independent binomials. 
We first introduce appropriate graph theoretic notation. Let $\ell, n$ be integers and let $S=[\ell]$ and $T=[n+\ell]\setminus [\ell]$. We use $S$ and $T$ as the two parts of the vertex set of a bipartite graph $G$, i.e.\  a graph $G$ with bipartition $(S, T)$. Such a graph is said to   have degree sequence $(\sv,\tv)$  if vertex $a$ has degree $s_a$ for all $a\in S$, and $v$ has degree $t_v$ for all $v\in T$. (Our convention is to denote elements of   $S$   by $\{a,b,\ldots\}$ and elements of $T$  by $\{v,w,\ldots\}$.)  We let $\cD(G)$ denote the degree sequence of $G$. 
When $\ell =n$, we use the fact that a digraph on $n$ vertices with out-degree sequence $\sv$ and in-degree sequence $\tv$ corresponds to a bipartite graph with degree sequence  $(\sv,\tv)$, the equivalence obtained by directing all edges from $S$ to $T$. For use in the digraph case, if $a\in S$ we define   $\mate{a}=a+\ell\in T$, and for $v\in T$ we define   $\mate{v}=v-\ell\in S$. The digraph contains a loop if and only if the bipartite graph has an edge joining $a$ to $a'$.

 The following probability spaces play an important role in this paper. 
Let $\G(\ell,n,m)$ denote the bipartite graph chosen uniformly at random among all bipartite graphs with bipartition $(S,T)$ and with $m$ edges. 
In the case when $\ell=n$, conditioning on the event that none of those $m$ edges is of the form $aa'$ yields a model of random directed graphs  without loops which we call  $\vec \G(n,m)$. 
We define $\cD(\G(\ell,n,m))$ and $\cD(\vec \G(n,m))$ to be the corresponding probability spaces of degree sequences of $\G(\ell,n,m)$ or of $\vec\G(n,m)$, respectively.
Let $\cB_p(\ell,n)$ be the probability space of vectors of length $\ell+n$ where the first $\ell$ elements are distributed as $\Bin(n,p)$ and the next $n$ are distributed as $\Bin(\ell ,p)$. Furthermore, let $\cB_m(\ell,n)$ be the restriction of $\cB_p(\ell,n) $ to the event $  \Sigma_1=\Sigma_2= m $, where $\Sigma_1$ is the sum of the first $\ell$ elements of the vector, and $\Sigma_2$ the sum of the other $n$ elements. 
Similarly, define $\vec\cB_p(n)$ to be the probability space of random vectors of length $2n$, every component being independently distributed as $\Bin(n-1,p)$. Finally, let $\vec\cB_m( n)$ be the restriction of $\vec\cB_p( n)$ to the event $\Sigma_1=\Sigma_2= m $, where $\Sigma_i$ is defined as above with $\ell=n$. 
Note that if  $\sum s_a = \sum t_v=m$, then  
\begin{align}\lab{probbin}
\pr_{\cB_m(\ell,n)}(\sv,  \tv)  &=  {\ell n \choose m }^{-2} \prod_{a\in S} { n \choose s_a }\prod_{v\in T} { \ell \choose t_v}  \mbox{ and}\nonumber\\
 \pr_{ \vec\cB_m( n)} (\sv,  \tv)  &=  { n(n-1) \choose m}^{-2} \prod_{a\in S} { n-1 \choose s_a }\prod_{v\in T} { n-1 \choose t_v},   
\end{align}
which we note are both independent of $p$. 

Our main result for degree sequences of ``medium density" states essentially that for certain sequences $\dv$, the probability $\Pr_{\cD(\G)}(\dv)$ is asymptotically equal to $\Pr_{\cB_m}(\dv)\correct(\dv)$, where $\G =\G(\ell,n,m)$ and $\cB_m =\cB_m(\ell,n)$ in the bipartite case, $\G =\vec\G(n,m)$ and $\cB_m =\vec\cB_m(n)$ in the digraph case, and where $\correct$ is a correction factor which we define next. For asymptotics in this paper, we take $n\to\infty$; the restrictions  on $\ell$ will also  ensure that $\ell\to\infty$. 

With $S$ and $T$ as above, let $\dv$ be a sequence of length $N=\ell+n$. We set $M_1=M_1(\dv) =  \sum_{i=1}^N d_i$ and use $\sv=\sv(\dv)$ and $\tv=\tv(\dv)$ to denote the vectors consisting of the first $\ell$, and  of the last $n$, entries of $\dv$ respectively. Thus, $\dv=(\sv,\tv)$. We also let $\savg=\savg(\dv)$ and $\tavg=\tavg(\dv)$ denote average of the components of $\sv$, and of $\tv$, respectively, that is 
$\savg = \frac{1}{\ell}\sum_{a\in S} s_a$ and $\tavg = \frac{1}{n}\sum_{v\in T} t_v$.
 Then we set 
$$\sigma^2(\sv)= \frac{1}{\ell}\sum_{a\in S} (s_a-\savg)^2, \quad \sigma^2(\tv)= \frac{1}{n}\sum_{v\in T} (t_v-\tavg)^2, $$
 and, in the digraph case, 
$$\sigma(\sv,\tv)=\frac{1}{n}\sum_{a\in S} (s_a-\savg)(t_{a'}-\tavg).$$  
We unify our analysis  of the two cases, bipartite graphs and digraphs, by introducing the indicator variable $\delta^{\di}$ which is $1$ in the digraph case (in which case $\ell = n$ is assumed) and $0$ in the bipartite case (in which case terms containing $\delta^{\di}$ as a factor may be undefined). This significantly simplifies notation  and permits us to emphasise the similarities between the two cases. 
 Define $\mu=\mu(\dv)=M_1(\dv)/(2n(\ell-\delta^{\di}))$. (This will denote the {\em relative edge density} of a bipartite graph or a digraph with degree sequence $\dv$.)
We then set 
\begin{align}\lab{corrH}
\corr{\dv} = \exp\left(-\frac12 \left(1-\frac{\sigma^2(\sv)}{s(1-\mu)}\right)
                     \left(1-\frac{\sigma^2(\tv)}{t(1-\mu)}\right) 
                     -\frac{\delta^{\di}\sigma(\sv,\tv)}{s(1-\mu)}\right)
\end{align}
for a sequence $\dv$ of length $\ell+n$, where $\mu = \mu(\dv)$, $\sv=\sv(\dv)$, $\tv=\tv(\dv)$. 
We can now state our main result.
\begin{thm}\thlab{t:mainbip}
For  a sufficiently small constant  $\mu_0$, the following holds. Let $1/2<\degspread<3/5$. Let $\ell$, $n$ and $m$ be integers that satisfy 
$$
m/(n\ell)<\mu_0, \quad   (\ell+n)^{5-5\degspread} 
  =o(\ell nm^{3-5\degspread}), 
$$
  and for all fixed $K>0$, $ \ell \log^K n+ n\log^K\ell  =o(m)$.  
Let $\D$ be the set of sequences $\dv=(\sv,\tv)$ with $\sv$ and $\tv$ of lengths $ \ell$ and $n$ respectively, satisfying $M_1(\sv)=M_1(\tv)=m$,  $|s_a-s|\le s^{\degspread}$ and $|t_v-t|\le t^{\degspread}$ for all $a\in S$ and all $v\in T$, where $\savg=m/\ell$ and $\tavg=m/n$.
Either set $\G=\G(\ell,n,m)$ and $\cB_m= \cB_m(\ell,n)$ (the bipartite case), or  set $\G=\vec \G(n,m)$ and $\cB_m= \vec\cB_m(n)$ and restrict to $\ell=n$ (the digraph case). 
 Then uniformly for all $\dv\in\D$,  \bel{enumFormula}
 \pr_{\cD(\G)} (\dv) =  \pr_{\cB_m}(\dv)\corr{\dv}\left(1+O\left(\frac{\log^2\ell}{\sqrt{\ell}}+\frac{\log^2n}{\sqrt{n}}+(\min\{s,t\})^{5\degspread-5}m^2/\ell n \right)\right).
 \ee
\end{thm}
Recall that in this paper, asymptotic statements refer to $n\to\infty$. The condition $ \ell \log^K n+ n\log^K\ell  =o(m)$, however, together with the trivial upper bound $m\le n\ell$ implies that $\ell\to\infty$ as well.  
We prove this theorem in Section~\ref{s:denseBip}. 
\begin{remark}\lab{r:first}
In view of~\eqn{probbin} and  the fact that $|\G(\ell,n,m)|  =  {\ell n \choose m }$, the formula in \thref{t:mainbip} is equivalent to the assertion that the number of bipartite graphs with degree sequence  $(\sv,\tv)$ is
$$
     {\ell n \choose m }^{-1}\prod_{a\in S}{ n \choose s_a }\prod_{v\in T} { \ell \choose t_v} 
 \exp\bigg(-\frac12 \bigg(1-\frac{\sigma^2(\sv)}{\mu(1-\mu) n}\bigg)
                     \bigg(1-\frac{\sigma^2(\tv)}{\mu(1-\mu)\ell  }\bigg) +O(\xi)\bigg),
$$ 
where $\xi$ is the error term from~\eqn{enumFormula}.
Similarly, \eqn{probbin} and the fact that $|\vec \G(n,m)| =    { n(n-1) \choose m }$ 
gives an asymptotic formula for the number of directed graphs with given degree sequence of in- and out-degrees. 
\end{remark}

Our corresponding result for the sparse case is the following. Although it is not new in the bipartite case (see below), it completes the full range of densities (in a sense, for instance, regarding regular digraphs) in the digraph case. For a sequence $\dv=(\dvA,\dvB)$ as above, define $\DeltaA=\Delta(\sv)=\max_{a\in S}(s_a)$, and similarly $\DeltaB=\Delta(\tv) =\max_{v\in T}(t_v)$. 

\begin{thm}\thlab{t:sparseCaseBip} 
Let  $0<\eps<1/2$, let $\ell$, $n$ and $m$ be integers 
such that $n/\log^4n + \ell/\log^4\ell=o(m)$, and set $s=m/\ell$, and $t=m/n$. 
Let $\D$ be a set of  sequences $\dv=(\sv,\tv)$ such that $\sv$ and $\tv$ have length $\ell$ and $n$, respectively, and such that $M_1(\dvA)=M_1(\dvB)=m\ge 1$ and $\DeltaA^3\DeltaB^3 (n\ell)^{\eps/2} = o(\min\{sm, tm\})$ uniformly over $\D$. 
Either set $\G=\G(\ell,n,m)$ and $\cB_m= \cB_m(\ell,n)$ (the bipartite case), or  set $\G=\vec \G(n,m)$ and $\cB_m= \vec\cB_m(n)$ and restrict to $\ell=n$ (the digraph case). Then uniformly for $\dv\in\D$, 
$$
\pr_{\cD(\G)}(\dv) =  \pr_{\cB_m}(\dv)\corr{\dv}\bigg(1+O\bigg(\frac{\DeltaA^3\DeltaB^3 (n\ell)^{\eps/2}}{m} ( 1/\savg +1/\tavg) +   n^{\eps-1/2}  +  \ell^{\eps-1/2} \bigg)\bigg). 
$$ 
\end{thm}

We prove this theorem in Section~\ref{s:sparseBip} before tackling the more involved case of medium range densities. We note at this point that if we restrict $\D$ to the set of sequences where $s_a,t_v\ge 1$ for all $a\in S$ and $v\in T$ then  $m\ge n, \ell$ and so the condition 
$n/\log^4n + \ell/\log^4\ell=o(m)$ is always true (as $n$   tends to infinity), and the condition 
 $\DeltaA^3\DeltaB^3(n\ell)^{\eps/2} =o(\min\{sm,tm\})$ is implied by $\DeltaA^3\DeltaB^3 = o(m^{1-\eps})$ as $s,t\ge 1$.  
Sequences failing these conditions therefore contain entries 0, which are less interesting since the formula is then often implied by considering only the non-zero entries. One could also apply our method to reach further into this very sparse case, but given these considerations, it is possibly not warranted, and we do not attempt to do so here. Similarly, further examination of our argument should yield results covering cases with wider disparities between $\ell$ and $n$. 

There have been many contributions to this topic in the past. Finding (asymptotic) formulae for the number of bipartite graphs with a given degree sequence goes back to Read's thesis~\cite{read1958} and gained wider interest since the 1970's, including~\cite{ bbk1972, b1974, bm1986,CM, es1971,  m1984, MWX,  mp1976, o1969, wormald1980}. 
In particular, the sparse case is best covered by Greenhill, McKay and Wang~\cite{GMW}, who proved an asymptotic formula for the number of bipartite graphs of a given sequence $(\sv,\tv)$, provided that $M_1(\sv)=M_1(\tv)$ and $\sv_{\max}\tv_{\max} =o\left((M_1(\sv))^{2/3}\right)$, and their result covers the bipartite version of \thref{t:sparseCaseBip}, in terms of both the density range $m/n\ell$ and the size of the error terms. 
This  
is supplemented by formulae for the number of dense bipartite graphs with specified degree sequences by Canfield, Greenhill, and McKay~\cite{CGM} that apply as long as $\ell$ and $n$ are not too far apart. 
In fact, in \cite{CGM} it was found that the formulae for the sparse and the dense case can be unified to produce the formula in \thref{t:mainbip}, which was implicitly conjectured in~\cite{CGM} to hold for the cases in between. This conjecture is essentially verified by  \thref{t:mainbip} for a wide range of parameters $\sv$ and $\tv$.

A special case is that of so-called semi-regular bipartite graphs, in which all vertices on one side of the bipartition have degree $s$, say, and all vertices on the other side have degree $t$. So let   $\sv$ denote the constant vector of length $\ell$ in which every entry is $s$, and  $\tv$ denote the constant vector of length $n$ in which every entry is $t.$ In 1977, Good and Crook~\cite{gc1977} suggested that the number of bipartite graphs with degree sequence $(\sv,  \tv)$
is roughly 
$ 
 \binom{n}{s}^{\ell}\binom{\ell}{t}^{n}/\binom{\ell n}{m} 
$
when $m= s\ell = tn$.
Some of the references mentioned above, in  particular~\cite{MWX} and~\cite{CM}, verify that this formula is correct up to a constant factor, for  particular ranges of $m$, $n$, $s$ and $t$,  by showing that the number is 
\bel{suggbip}
\frac{\binom{n}{s}^{\ell}\binom{\ell}{t}^{n}}{\binom{\ell n}{m}}  e^{-1/2+o(1)}.  
\ee
 This asymptotic assertion is immediately equivalent to
$\pr_{\G(\ell,n,m)} (\sv,  \tv) \sim \pr_{  \cB_m(\ell,n)}(\sv,  \tv)\correct(\sv,  \tv)$.  Consequently, \thref{t:mainbip} verifies~\eqn{suggbip} for a new range  of parameters in the moderately dense case.

For digraphs without loops, there are far fewer corresponding results. For the dense case, i.e.~when the number of edges is $\Theta(n^2)$, a result by 
Greenhill and McKay~\cite{GM} implies an asymptotic formula.
Barvinok~\cite{b2010} provides upper and lower   bounds  which are coarser but their bounds apply to a wider range of in- and out-degree sequences. 
The only result we are aware of that explicitly enumerates loopless digraphs by degree sequence in the sparse case is by Bender~\cite{b1974}, which only applies for bounded degrees. However, it is clear that the standard   techniques used previously for sparse graph enumeration could be used to increase the density and obtain results more in line with the existing ones for bipartite graphs.

 {\subsection{Models for the degree sequences of   random graphs }
In 1997, McKay and Wormald~\cite{degseq1} showed that if a certain enumeration formula holds for the number of graphs of a given degree sequence then the degree sequences of the random graph models $\Gnm$ and $\Gnp$ can be modelled by certain binomial-based models. The model for $\Gnp$ showed that the degree sequence was distributed almost the same as a sequence of independent binomial random variables, subject to having even sum, but with a slight twist that introduces dependency. It was also shown there that for properties of the degree sequence satisfying some quite general conditions, this conditioning and dependency make no significant difference, and hence those properties are essentially the same as for a sequence of independent binomials.

At that time, the existing formulae for the sparse and the dense case supplied that relationship of the models. Recently, the enumeration results of \cite{lw2018} for the medium range provide the missing formulae for the gap range of densities, establishing a conjecture from~\cite{degseq1}. 
A natural supposition since~\cite{degseq1} appeared was that the degree sequences of random bipartite graphs and digraphs satisfy similar properties. This was an implicit conjecture of McKay and Skerman~\cite{MS}, who adapted some of the arguments in~\cite{degseq1} to show that the existing enumeration results for dense bipartite graphs and directed graphs imply a binomial-based model of the degree sequences of such graphs. This is quite analogous to the model in the graph case, except that it contains an extra complicating conditioning required because the sum of degrees of the vertices in each part must be equal.
McKay and Skerman point out that, once the  enumeration formulae are proved in the missing ranges, one would expect the model results to follow. 
Our enumeration results stated above provide what is necessary to immediately establish the relevant conjecture in the case of $\G(\ell,n,m)$ and $\vec\G(n,m)$, as described below, provided that $\ell$ and $n$ are not too disparate. For their binomial random graph siblings $\G(\ell,n,p)$ and $\vec\G(n,p)$, in which edges are selected independently with probability $p$, one would expect that arguments similar to those in~\cite{MS}, in conjunction with our results, will now suffice. 
 
Let $A_n$ and $B_n$ be two sequences of probability spaces with the same underlying set for each $n$. Suppose that whenever a sequence of events $H_n$ satisfies $\Pr(H_n)=n^{-O(1)}$ in either model, it is true that $\Pr_{A_n}(H_n)\sim\Pr_{B_n}(H_n)$, where by $f(n)\sim g(n)$ we mean that $f(n)/g(n)\to 1$ as $n\to\infty$. Then we call $A_n$ and $B_n$ {\em asymptotically quite equivalent (a.q.e.).}  We use $\omega$ to mean a function going to infinity as $n\to\infty$, possibly different in all instances.

\begin{thm}\thlab{t:bipmodel} 
\begin{enumerate}[label=(\alph*)]
\item\label{model-a}  The probability spaces  $\cD(\vec \G(n,m))$ and $\vec \cB_m(n) $ are a.q.e.~provided that  $\log ^3n  =  o(\min\{m, n\ell-m\})$;
\item\label{model-b} The probability spaces $\cD(\G(\ell,n,m))$  and $\cB_m(\ell,n) $ are a.q.e.~provided that $\max\{m, n\ell-m\}=\omega \log n$ and at least one of the following holds: 
\begin{enumerate}[label=(\roman*)]
\item\label{model-b-ii}
$\ell\le n$ and for some fixed $\mu_0>0$ and $\eps>0$ 
we have 
$m<\mu_0n\ell$ and $n^3=o(\ell^2m^{1-\eps})$;
\item\label{model-b-iii}
 $n/\log^4n + \ell/\log^4\ell=o(m)$ and  for some fixed $\eps>0$ we have $m^{4+\eps}=o(n^2\ell^2\min\{ \ell,n\})$;
\item\label{model-b-iv}
$m=o(\min\{n/\log^2n,   \ell/\log^2\ell\})$ and $\log^3 \ell + \log^3 n=o(m)$.
\end{enumerate}
  \end{enumerate}
 \end{thm}
 We prove this theorem in Section~\ref{s:denseBip}. 
 We note that the assertion for~\ref{model-a} for the range $m> n^2 / \log n$ is covered by McKay and Skerman~\cite[Theorem 1(d)]{MS}. For the bipartite case~\cite[Theorem 1(c)]{MS}  covers the range  $m>\ell n / \log n$ and $\ell = n^{1+o(1)}$.  \thref{t:bipmodel}~\ref{model-b}\ref{model-b-ii} applies for a slightly larger range of $\ell$ and $n$, at least for large-ish $m$. 
The last condition in~\ref{model-b-ii}  is equivalent to  
$$\frac{n^{2+\eps}}{\ell^{3-\eps}} \ll \mu^{1-\eps}$$ for some $\eps>0,$ where $\mu = m/n\ell$. Thus, $n$ may be as large as $\ell^{2-\eps}$ for sufficiently large density $\mu$. 

Finally, we note that when $\min\{\ell,n\}\gg \max\{\ell,n\}^{10/11+\eps}$ for fixed $\eps>0,$ then all values of $m$ are covered by \thref{t:bipmodel} (swapping $\ell$ and $n$ in~\ref{model-b}\ref{model-b-ii} if necessary) and using~\cite[Theorem 1(d)]{MS} for the dense cases of both~\ref{model-a} and~\ref{model-b}.

\subsection{ Edge probabilities.}
As a by-product of our proof of \thref{t:mainbip} in Section~\ref{s:denseBip}, we obtain asymptotic formulae for the edge probabilities in a random bipartite graph with a given  degree sequence, and of a random digraph with a given sequence of out- and in-degrees.

\begin{thm}\thlab{t:edgeprobability} 
Let $n$, $\ell$, $m$, and $\D$ be as in \thref{t:mainbip}  and  let $\cG = \cG(\ell,n,m)$ or $\cG = \vec\cG(n,m)$. Let $a\in S$ and $v\in T $, with $v\ne a'$   in the digraph case. Then uniformly for $\dv=(\sv,\tv)\in \D$, the probability that  $av$  is an edge of  $G\in \cG$, conditional on the event that $\cD(G)=\dv$, is
$$
  \frac{s_at_v}{m-\delta^{\di}\tavg}\left(1-\frac{(s_a-\savg)(t_v-\tavg)}{m-\delta^{\di}\tavg-\tavg\savg}
+\frac{(s_a-\savg)\sigma^2(\tv)}{\tavg\savg(\ell-\tavg)}
+\frac{(t_v-\tavg)\sigma^2(\sv)}{\tavg\savg(n-\savg)}
+\frac{\delta^{\di}(t_{\mate{a}}+s_{\mate{v}})}{n-1} +O\left(\min\{s,t\}^{4\degspread-4}\frac{m}{n\ell}\right)
 \right),
$$
where $\savg= m/\ell$ and $\tavg = m/n$. 
\end{thm}
 We prove this theorem in Section~\ref{s:denseBip}.




\section{Preliminaries}
\lab{s:prelim}

As we indicated in the introduction, the argument in this paper derives from that  in~\cite{lw2018}, whose notation and structure we will follow quite closely. 
Differences occur though to account for the fact that we are dealing with certain forbidden edges. 
Naturally, we resort to notation used in~\cite{lw2018} and add notation that is special to the bipartite case. We then state several intermediate results from~\cite{lw2018}. 

\subsection{Notation}\lab{s:notation} 
Our {\em graphs} are simple, that is, they have no loops or multiple edges. We write $a\sim b$ to mean that $a/b\to 1$, $f=O(g)$ if $|f|\le Cg$ for some constant $C$,  and $f=o(g)$ if $f/g\to 0$.  We use $\omega$ to mean a function going to infinity, possibly different in all instances. Also $\binom{V}{2}$ denotes the set of $2$-subsets of the set $V$, and $V$ is often of the form $[N]$, which denotes $\{1,\ldots , N\}$. In this paper  multiplication by juxtaposition has precedence over ``$/$", so for example $j/\mu N^2= j/(\mu N^2)$. 

Let $N$ be an integer and let $V=[N]$. 
Assume that $\cA=\cA(N)\se \binom{ [N] }{2}$ is specified; we call this the set of {\em allowable pairs}. Note that as usual we regard the edge joining vertices $u$ and $v$ as the unordered pair $\{u,v\}$, and denote this edge by $uv$ following standard graph theoretic notation. 
A sequence $\dv= (d_{1},\ldots,d_{N})$ is called  
{\em $\cA$-realisable} 
if there is a graph $G$ 
on vertex set $V$ such that vertex $a\in V$ has degree $d_a$ and all edges of $G$ 
are allowable pairs. 
In this case, we say $G$ {\em realises $\dv$ over $\cA$}.  In standard terminology, if $\dv$ is $\binom{V}{2}$-realisable, it is  {\em graphical}. 
Let $\cG_{\cA}(\dv)$ be the set of all graphs that realise $\dv$ over $\cA$. 
The graph case when $\cA=\binom{V}{2}$ is dealt with in~\cite{lw2018}. 
In this paper, we are particularly interested in the following two special cases of $\cA$. 
\begin{itemize}
\item  {\bf Bipartite graph case.\\} 
Let $ \ell, n$ be integers and set $N =\ \ell +n$. 
Set $\cA=\cA^{\bi} = \{uw:u\in [ \ell], w\in  [N]\sm[\ell]\}$. 
Then $\cG_\cA(\dv)$ is the set of all bipartite graphs $G$ on vertex set $[N]$ that 
realise the degree sequence $\dv=(\sv,\tv)$ with one part being $S=[\ell]$ and the other part $T=[N]\setminus[\ell]$. 

\item {\bf Digraph case.\\} 
Assume that $N$ is even and let $n$ be an integer such that $N =2n$. 
Set $\cA=\cA^{\di} = \{uw: u\in [n], w\in [n+1,2n], u+n \neq w\}$. 
Then $\cG_\cA(\dv)$ corresponds to the set of all bipartite graphs $G$ on vertex set $[2n]$ that 
realise the degree sequence $\dv=(\sv,\tv)$ with one part being $S=[n]$ and the other part $T=[2n]\setminus[n]$ that do not contain any edge of a predefined matching, or equivalently,  
$\cG_\cA(\dv)$ corresponds to the set of all digraphs $G$ on vertex set 
$[n]$ that have no loops and that realise the out-degree sequence $\sv$ and in-degree sequence $\tv$. Recall that $\mate{a}=a+n$ for $a\in S$ and $v'= v-n$ for $v\in T$ so that edges of the form $aa'$ are forbidden. 
\end{itemize} 

Let $\ell$, $n$ be integers and suppose that $\dv$ is a sequence of length  $\ell + n$. 
Recall the definitions of  $\sv=\sv(\dv)$, $\tv=\tv(\dv)$, $\savg$, $\tavg$, $M_1(\dv)$, $\sigma(\sv,\tv)$, and of $\sigma^2(\dv)$ from the introduction. 
We also use $\Delta$ or $\Delta(\dv)$ to denote $\max_i d_i$, in line with the notation for maximum degree of a graph. 
With $\cA$ understood (to be either $\cA^{\bi} $ or $\cA^{\di} $ in this paper) we write $\mu=\mu(\dv)	$ for the quantity $M_1(\dv)/|\cA|$ and note that this agrees with the definition of $\mu$ given just above~\eqref{corrH} in the introduction. 
Throughout this paper we use ${\bf e}_i$ to denote the elementary unit vector with 1 in its  coordinate indexed by $i$. 
We say $\dv$ is  {\em balanced}   if $M_1(\sv)=M_1(\tv)$. Clearly being balanced is necessary for $\dv$ to be $\cA$-realisable in either of the cases $\cA=\cA^{\bi}$ or $\cA=\cA^{\di}$. Furthermore, we say that $\dv$ is {\em $S$-heavy} if $M_1(\sv)=M_1(\tv)+1$, and we call it {\em $T$-heavy}  if $M_1(\sv)=M_1(\tv)-1$.  

Finally, we use $1\pm \xi$ to denote a quantity between $1-\xi$ and $1+\xi$ inclusively.
\subsection{Cardinalities, probabilities and ratios} 
We first quote a simple result by which we leverage absolute estimates of probabilities from comparisons of related probabilities.
\begin{lemma}[Lemma 2.1 in \cite{lw2018}]\thlab{l:lemmaX} 
Let  $\cS$ and $\cS'$ be probability spaces with the same underlying set $\Omega$.  
Let $G$ be a graph with vertex set  $\W\subseteq \Omega$  such that $ \pr_\cS(v),\pr_{\cS'}(v) >0$  for all $v\in \W$. Suppose that  $\eps_0,  \delta >0$ such that 
 $\min\{ \pr_\cS(\W), \pr_{\cS'}(\W)\}>1-\eps_0>1/2$,  and such that  for every edge $uv$ of $G$,  
$$
\frac{\pr_{\cS'}(u)}{\pr_{\cS'}(v)}=   e^{O( \delta )}  \frac{\pr_{\cS}(u)}{\pr_{\cS}(v)}   
$$
where the constant implicit in $O(\cdot)$ is  absolute. 
Let $r$ be an upper bound on the diameter of $G$ and assume $r<\infty$. Then for each $v\in \W$ we have
$$
  \pr_{\cS'}(v)  =e^{O(r \delta +\eps_0)}  \pr_\cS(v) ,
$$
with again a   bound uniform for all $v$.
\end{lemma}
Using the lemma calls for analysing the ratios of probabilities both in the ``true'' probability space $\cS'$ (which will be the degree sequence of $\G(\ell,n,m)$ or of $\vec\G(n,m)$) and in an ``ideal'' probability space $\cS$ by which we are approximating the true space. 
This leads to computing ratios of closely related instances of the expression on the right hand side of~\eqref{enumFormula}. 
Let $\dv=(\sv,\tv)$ be a sequence where $\sv$ and $\tv$ are of length $\ell$ and $n$, respectively, let $a, b\in S$, and assume that $\dv$ is $S$-heavy. 
Note first that the following are immediate from~\eqref{probbin}:
$$\frac{\Pr_{\cB_m(\ell,n)}(\sv-\ea,\tv)}{\Pr_{\cB_m(\ell,n)}(\sv-\eb,\tv)} 
	= \frac{s_a(n+1-s_b)}{s_b(n+1-s_a)} \quad \text{ and }\quad
\frac{\Pr_{\vec\cB_m(n)}(\sv-\ea,\tv)}{\Pr_{\vec\cB_m(n)}(\sv-\eb,\tv)} 
	= \frac{s_a(n-s_b)}{s_b(n-s_a)}.$$ 
Similarly,  straight from the definition of $\correct$ in~\eqref{corrH} we have 
\begin{align*} 
\frac{\corr{\sv-\ea,  \tv}}{\corr{\sv-\eb,  \tv}} 
 &= 
  \exp\bigg(\frac{s_b-s_a}{s\ell(1-\mu')}
 \left(1-\frac{\sigma^2(\tv)}{t(1-\mu')}\right)
 +\frac{ \delta^{\di} (t_{a'}-t_{b'})}{sn(1-\mu')}\bigg),  
\end{align*} 
where $\mu'=\mu(\sv-\ea,\tv)$, $s=s(\sv-\ea,\tv)$, $t=t(\sv-\ea,\tv)$, 
which are, in this case, equal to $\mu(\sv-\eb,\tv)$, $s(\sv-\eb,\tv)$, and $t(\sv-\eb,\tv)$, respectively (recalling that $\mu$ is slightly different in the two cases of $\cA^{\bi}$ and $\cA^{\di}$), and where, we recall, $\delta^{\di}$ is the indicator variable for the digraph case. 
Therefore, denoting by $\apx(\dv')$ the function $\Pr_{\cB_m}(\dv')\correct(\dv')$, 
we get a ``combined goal ratio'' in the two cases which is 
\begin{align}\lab{H}
\frac{\apx(\sv-\ea,\tv)}{\apx(\sv-\eb,\tv)} 
 =\frac{s_a(n+\delta^{\bi} -s_b)}{s_b(n+\delta^{\bi}-s_a)}
 \exp\bigg(\frac{s_b-s_a}{s\ell(1-\mu')}
 \left(1-\frac{\sigma^2(\tv)}{t(1-\mu')}\right)
 +\frac{ \delta^{\di} (t_{a'}-t_{b'})}{sn(1-\mu')}\bigg),  
\end{align}
where $\delta^{\bi} = 1-\delta^{\di}$. 

To analyse the ratios of such nearby sequences in the ``true'' probability space note that, with the above notation, $\Pr_{\cD(\cG)}(\dv)$ in \thref{t:mainbip} is just $|\cG_{\cA}(\dv)|/|\cG|$ where $\cG$ is the random graph space $\cG(\ell,n,m)$ or $\vec\cG(n,m)$. Let us introduce some more notation. 
Let $F\se \cA$, i.e.~a subset of the allowable edges. 
We write $\Num_F(\dv)$ and $\Num^*_F(\dv)$ 
for the number of graphs $G\in \cG_{\cA}(\dv)$ 
that contain, or do {\em not} contain, the edge set $F$, respectively.  (When  $\Num$ and similar notation is used, the set $\cA$ should be clear by context.)
We abbreviate $\Num_F(\dv)$ to $\Num_{ab}(\dv)$   if  $F={\{ab\}}$ (i.e.~contains the single edge $ab$), 
and put $\Num(\dv)=|\cG_{\cA}(\dv)|$. Additionally, for a vertex $a\in V$, we set $\cA(a) =\{v\in V: av \in \cA\}$, and, with $\dv$ understood, we use $\cA^*(a)$ for the set of $v\in \cA(a)$ such that $\Num_{av}(\dv)>0$. 

We pause for a notational comment. In this paper, a subscript  $ab$  is always  interpreted as an ordered pair $(a,b)$ rather than an edge (and similar for triples). This is irrelevant for $\Num_{ab}(\dv) =  \Num_{ba}(\dv)$ since the two ordered pairs signify the same edge, but the distinction is important with other notation such as the following. 
For vertices $a,b \in V$, if $\dv$ is a sequence such that $\dv-\eb$ is $\cA$-realisable,   
we define 
\bel{realRatio}
R_{ab}(\dv)  =\frac{\Num(\dv-\ea)}{\Num(\dv-\eb)} 
\ee
and note that this is exactly $\Pr_{\cD(\cG)}(\dv-\ea)/\Pr_{\cD(\cG)}(\dv-\eb)$. 
Estimating those ``true'' ratios will be tightly linked to estimating the following. 
For  $F\se \cA$, let 
$$P_F(\dv)  =\frac{\Num_{F}(\dv)}{\Num(\dv)},$$ 
which is the probability 
that the edges in $F$ are present in a graph $G$ that 
is drawn uniformly at random from $\cG_{\cA}(\dv)$. 
Of particular interest are the probability of a single edge $av$ and a path $avb$, 
for which we simplify the notation to
\bel{realProb}
P_{av}(\dv) =P_{\{av\}}(\dv),\qquad P_{avb}(\dv) =P_{\{av, bv\}}(\dv).
\ee

The following is~\cite[Lemma 2.2]{lw2018}, used to switch between degree sequences of differing total degree.
\begin{lemma}\thlab{trick17}
Let $a v\in \cA$ and let $\dv$ be a sequence of length $N$. 
Then 
\begin{align*}
\Num_{av}(\dv) 
	&= \Num (\dv - \ea - \ev) -  \Num_{av} ({\bf d} - \ea - \ev)\\  
	&= \begin{cases} 
		\Num ({\bf d} - \ea - \ev) (1- P_{av}(\dv- \ea - \ev)) & \mbox{if } \Num({\bf d} - \ea - \ev) \ne 0\\ 
		0 & \mbox{otherwise.}
	\end{cases}
\end{align*}
\end{lemma}

In Lemma 2.3 in~\cite{lw2018} we bound the probability of an edge of a random graph in $\cG(\dv)$ in the graph case. A similar switching argument is used to obtain corresponding bounds in the bipartite and digraph cases. 
Recall that by $\Delta(\dv)$ we denote $\max_i d_i$, and that $M_1(\dv) = \sum_i d_i$. 
   
\begin{lemma}\thlab{l:simpleSwitching}
Let $\cA$ be $\cA^{\bi}$ or $\cA^{\di}$ and let                 
$(\sv,\tv)$ be an $\cA$-realisable sequence.  
Then for any  $av\in\cA$  
we have 
$$
 P_{av}(\sv,\tv)\le \frac{\Delta(\sv)\Delta(\tv)}{M_1(\sv) 
 \left(1- 2(\Delta(\sv)+1)(\Delta(\tv)+1)/M_1(\sv)\right)}. 
$$
\end{lemma}
\begin{proof} 
 Assume without loss of generality that $a\in S$ which forces $v\in T$ in both the digraph and bipartite cases. 
For each bipartite graph $G$ with degree sequence $(\sv,\tv)$ and an edge joining $a$ and $v$, we can perform a switching (of a type often used previously in graphical enumeration) by removing both $av$ and  another randomly chosen edge $bw$ (with $b\in S$, and $w\in T$), and inserting the edges  $aw$ and $bv$, provided that no multiple edges are formed. Note that the way we choose $b$ and $w$ no loops can occur this way. In the digraph case, we should also make sure that $w\neq\mate{a}$ 
and that $b\neq v'$, 
since the pairs $aa'$ and $vv'$ are not allowable. 
The number of such switchings that can be applied to $G$ with the vertices of each edge ordered, is at least 
$$
M_1(\sv) - (\Delta(\sv) +1)\Delta(\tv)- \Delta(\sv)(\Delta(\tv)+1) 
$$
since there are $M_1(\sv)$ ways to choose $b\in S$ and $w\in T$, 
whereas the number of such choices that are ineligible is at most the number of choices with $b$ being a neighbour of $v$ (which automatically rules out  $b=a$) or $b=\mate{v}$, or similarly for $w$. On the other hand, for each graph $G'$ in which  $av$ is {\em not} an edge, the number of ways that it is created by performing such a switching backwards is at most $\Delta(\sv)\Delta(\tv)$. Counting the set of all possible switchings over all such graphs $G$ and $G'$ two different ways shows that the ratio of the number of graphs with  $av$  to the number without  $av$ is at most
$$
\beta:=\frac{\Delta(\sv)\Delta(\tv)}{M_1(\sv)- 2(\Delta(\sv) +1)(\Delta(\tv)+1) }.
$$
Hence $P_{av}(\dv)\le \beta/(1+\beta)$, and the lemma follows in both cases. 
\end{proof}

\subsection{Proof structure}\lab{s:template}

We recall the template of the method introduced in~\cite{lw2018}. We follow this template in both the sparse and dense cases. 

\noindent
{\bf Step~1.}~Obtain an estimate of the ratio $R_{ab}(\dv)$ between the numbers of graphs of related degree sequences, using the forthcoming Proposition~\ref{l:recurse}. This step is the crux of the whole argument. 

\smallskip
\noindent
{\bf Step 2.} By making suitable definitions, we cause this ratio  to appear  as   the expression $\pr_{\cS'}(\dv)/ \pr_{\cS'}(\dv')$ for some probability space $\cS'$ on an underlying set $\Omega$ in an application of  Lemma~\ref{l:lemmaX}. There, $\Omega$ is the set of degree sequences, with probabilities in $\cS'$ determined by  the random graph under consideration, and the graph $G$ in the lemma has a suitable vertex set $\W$ of such sequences. Each edge of $G$ is in general a  pair of degree sequences   $\dv-\ea$ and $\dv-\eb$ of the form occurring in the definition of $R_{ab}( \dv) $. Having defined $G$, we may call any two such degree sequences {\em adjacent}.  

\smallskip
\noindent
{\bf Step 3.}~Another probability space $\cS$ is defined on $\Omega$, by taking a probability space $\cB$ directly from a joint binomial distribution,   together with a function $\corr{\dv}$ that varies quite slowly, and defining  probabilities in $\cS$ by the equation
$  \pr_\cS(\dv)=    \pr_{\cB }(\dv) \corr{\dv} /\ex_{\cB} \correct$.

\smallskip
\noindent
{\bf Step 4.}~Using sharp concentration results, show that  $P(\W) \approx 1$ in both of the  probability spaces $\cS$ and $\cS'$ (where, by $\approx$, we mean approximately equal to, with some specific error bound in each case). As part of this, we show that $\ex_{\cB } \correct \approx 1 $. At this point, we may specify $\eps_0$ for the application of Lemma~\ref{l:lemmaX}.
 
\smallskip
\noindent
{\bf Step 5.}~Apply Lemma~\ref{l:lemmaX} and the conclusions of the previous steps to deduce   $P_{\cS'}(\dv) \approx P_\cS(\dv)\approx \pr_{\cB }(\dv) \corr{\dv}$. Upon estimating the errors in the approximations, which includes bounding the diameter of the graph $G$, we obtain an estimate for the probability $P_{\cS'}(\dv)$ of the random graph having degree sequence $\dv$ in terms of a known quantity.

\subsection{Realisability}
As in the graph case in~\cite{lw2018}, before estimating how many (bipartite) graphs have degree sequence $\dv$, for preparation we need to know that there is at least one such graph for various $\dv$. Mirsky~\cite[p.~205]{M} gives a necessary and sufficient condition  for the existence of a non-negative integer matrix with row and column sums in specified intervals. For the case that those sums are specified precisely, the statement is the following.
\begin{thm}[Corollary of Mirsky~\cite{M}]\lab{t:mirsky}  Let $0\le r_i$, $0\le c_j$, $m_{ij}\ge 0$ be integers for all $1\le i  \le\ell $, $1\le j\le n$ such that $ \sum_{1\le i\le  \ell  } r_i = \sum_{ 1\le j\le n} c_j$. Then there exists an $ \ell\times n$ integer matrix $B=(b_{ij})$ with row sums $r_1,\ldots, r_\ell $ and column sums $c_1,\ldots, c_n$ such that $0\le b_{ij}\le m_{ij}$ for all such $i$ and $j$ if and only if, for all $X\subseteq \{1,\ldots, \ell\}$ and $Y\subseteq  \{1,\ldots, n\}$,
$$
\sum_{i\in X,\,j\in Y} m_{ij} \ge \sum_{i\in X } r_i - \sum_{ j\notin Y} c_j.
$$
\end{thm}
We   use this to show existence of bipartite graphs with given degrees and forbidden edges for the cases of interest, avoiding maximum generality in order to keep it simple. In order to apply this to digraphs, one would set $\ell= n$ and regard the edges as directed from the first part to the second. For loopless digraphs, we merely forbid all edges of the form $\{i,i+n\}$. Recall that, with $\ell$ and $n$ understood, we set $S=[\ell]$ and $T=[n+\ell]\sm[\ell]$ for convenience.
\begin{lemma}\thlab{lem:bipRealisable}
Given  a constant $C\ge 1$, the following holds for $\ell,n$ sufficiently large and  $\eps>0$   sufficiently small.  Let  $s_a\ge 1$ and $t_v\ge 1$ be integers for all $a \in S$, $v\in T$, with  $m:= \sum_{a \in S} s_a = \sum_{ v\in T } t_v$.  Also let 
$F \se\{av : a\in S,v\in T\}$  be a set of unordered pairs, representing forbidden edges, with no more than $C$ pairs in $F$  containing any $w\in S\cup T$. Let $\Delta_S=\max_{a\in S} s_a $ and $\Delta_T=\max_{v\in T} t_v$. 
Then  there exists a bipartite graph  with bipartition $(S,T)$ with degrees $s_a$ for $a\in S$ and $t_v$ for $v\in T$, and containing no edge in the forbidden set $F$, provided that either of the following holds. 
\begin{enumerate}[label=(\alph*)]
\item  \label{lem:bipRealisablei} We have $m\le \ell n/9$, as well as  $\Delta_S \le 2s$  and  $\Delta_T  \le2t$ where   $s=m/\ell$ and $t=m/n$.
\item \label{lem:bipRealisableii} We have $\Delta_S\le \sqrt  m/2-C$ and  $\Delta_T\le \sqrt  m/2-C$.
\end{enumerate}
\end{lemma}
\begin{proof} 
We will apply Theorem~\ref{t:mirsky} with $m_{ij}=0$ if $\{i,j+\ell\}$ is a forbidden edge, and  $m_{ij}=1$ otherwise, and with $r_i=s_i$ and $c_j=t_{j+\ell}$. Note that 
$ 
\sum_{a\in X,v\in Y} m_{ij}\ge xy-C \min\{x,y\}
$ 
for all subsets $X\se S$, $Y\se T$, where $x=|X|$ and $y=|Y|$. 
We will show that 
for all $X\subseteq S$ and $Y\se T$, with $x=|X|$ and $y=|Y|$, we have
\bel{cond0}
 m-C \min\{x,y\}\ge  \sum_{a\in X} s_a +\sum_{v\in   Y} t_v -xy.
\ee
Equivalently, $ xy-C \min\{x,y\}\ge  \sum_{a\in X} s_a -\sum_{v\in T\sm Y} t_v$. Note that with the previous observation and Theorem~\ref{t:mirsky}, this implies that there is a matrix $B$ which is the adjacency matrix of the desired bipartite graph.

For (a), suppose first that $x\ge 2t+C$. Then using $\sum_{a\in X} s_a\le m$ and $\sum_{v\in   Y} t_v \le y\Delta_T \le 2y t$, we find that the right hand side of~\eqn{cond0} is at most $m-yC$, and~\eqn{cond0} follows. A symmetric argument works if $y\ge 2s+C$. So we may assume that neither of these occur. Then
$$
  \sum_{a\in X} s_a +\sum_{v\in   Y} t_v  \le x\Delta_S+y\Delta_T\le (2t+C)2s+(2s+C)2t
$$
and the left hand side is at least 
$m-C(x+y)/2\ge m- C(s+t)-C^2$.
Thus~\eqn{cond0}  follows if we show
that  $m\ge 8st+  5C(s+t)+C^2$, i.e.\ if $1\ge 8m/\ell n +  5C(1/\ell+1/n) +C^2/m$ (since $s=m/\ell$ and $t=m/n$). This  holds for  $\ell$ and $n$ sufficiently large because  $ n\le  m\le \ell n/9$. 

Now consider (b).
Without loss of generality, since $S$ and $T$ can be interchanged along with $X$ and $Y$, $\ell$ and $n$ etc., we assume $x\ge y$. First consider the case that $x\le \sqrt m / 2$. Then $y\le \sqrt m / 2$, and so  the right hand side of~\eqn{cond0} is  at most 
$$
x\Delta_S  +y\Delta_T   \le m/2-C(x+y),
$$
which implies~\eqn{cond0}. On the other hand, if $  x>\sqrt m / 2$ then we can bound the first summation in~\eqn{cond0} by $m$ and the second one by $y\Delta_T\le y\sqrt m/2-yC$, and use $xy> y\sqrt m/2$
to obtain~\eqn{cond0}.
\end{proof}


\subsection{Recursive relations}\lab{s:recursive}

In this subsection we collect results about recursive relations that were obtained in~\cite{lw2018}. The results were stated for an arbitrary set $\cA$ of allowable pairs in $\binom{V}{2}.$ 
Recall the definitions of the probabilities $P_{av}(\dv)$ and $Y_{avb}(\dv)$ in~\eqref{realProb}, and of the ratio $R_{ab}(\dv)$ in~\eqref{realRatio}.
\begin{proposition}[Proposition 3.1 in \cite{lw2018}]\thlab{l:recurse}
Let $\dv$ be a sequence of length  $n$ and let $\cA\se \binom{[n]}{2}$.
\begin{itemize}
\item[$(a)$] Let $a,v\in V$. If $\Num_{av}(\dv)>0$ 
then 
	\begin{equation*}
	P_{av}(\dv) =d_{v} \Bigg(\sum_{b\in \cA^*(v)}
	 R_{ba} ( \dv- \ev)
	\frac{1-P_{bv}(\dv - \eb - \ev)}
	{1-P_{av}(\dv - \ea - \ev)}\Bigg)^{-1}.
	\end{equation*}
\item[$(b)$] Let $a,b\in V$. If $\dv-\eb$ is $\cA$-realisable then 
	\begin{align} \lab{Ratformula}
	 R_{ab} ( \dv) &= \frac{d_{a}}{d_{b}}\cdot \frac{ 1-\Bad(a,b, {\bf d} - \eb)}{ 1-\Bad(b,a, {\bf d} - \ea)}, 
	\end{align}
	where 
	\begin{align}\lab{eq:bad}
	\Bad(i,j, \dw) 
	& = 
	 \frac{1}{d_{i}}\Bigg(\sum_{ v \in \cA(i)\setminus \cA(j)}  P_{iv}(\dw)   +
	 \sum_{ v \in \cA(i)\cap \cA(j) }  Y_{ivj}(\dw) \Bigg), 
	\end{align}
	provided that  $\Bad(b,a, {\bf d} - \ea) \ne 1$.

\item[$(c)$] Let $a,v,b$ be distinct elements of $V$. If $\dv-\ea-\ev$ is $\cA$-realisable then 
$$Y_{avb}(\dv)=\frac{P_{av}(\dv)\big(P_{bv}(\dv-\ea-\ev)-Y_{avb}(\dv-\ea-\ev)\big)}
	{1-P_{av}(\dv-\ea-\ev)}.$$
\end{itemize}
\end{proposition}


These recursive relations 
motivated the definition of operators in~\cite{lw2018} that we restate here.

Let  $\Z^N$  denote the set of non-negative integer sequences  of length $N$.   
For a given integer $N$ and a set $\cA\se \binom{[N]}{2}$ we define $\oA$ to be the set of ordered pairs $(u,v)$ with $\{u,v\}\in \cA$.  Ordered pairs are needed here because, although the functions of interest are symmetric in the sense that the probability of an edge $uv$ is the same as $vu$, our approximations to the probability do not obey this symmetry. 
Similarly, let $\oA_2$ denote the set of ordered triples $(u,v,w)$ with $u$, $v$ and $w$ all distinct and $\{u,v\}$, $\{v,w\} \in \cA$. 

Suppose we are given 
$\pv: \oA\times\Z^N\to\R_{\ge 0}$, 
$\yv: \oA_2\times \Z^N\to\R_{\ge 0}$ and 
$\rv: [N]^2 \times \Z^N\to\R_{\ge 0}$. 
We  write $\pv_{av}(\dv)$ for $\pv(a,v,\dv)$ (where $\dv\in \Z^N$), and remind 
the reader that in this paper, a subscript $av$ always denotes an ordered pair rather than an edge. 
Similarly, we write $\yv_{avb}(\dv)$ for $\yv(a,v,b,\dv)$ and 
$\rv_{ab}(\dv)$ for $\rv(a,b,\dv)$. 
We also define an associated function  $\bad(\pv,\yv)$ as follows. 
For $\dv\in \Z^N $ and $a,b\in [N]$ with $a\neq b$, set 
$\bad(\pv,\yv)(a,a,\dv) = 0$ and 
\begin{equation}\lab{def:bad}
\bad(\pv,\yv)(a, b, \dv) = 
\frac{1}{d_{a}}
\left( 
\sum_{v\in \cA(a)\sm \cA(b)} \pv_{av}(\dv) 
+ \sum_{v\in \cA(a)\cap \cA(b)} \yv_{avb}(\dv) \right).
\end{equation}
We define two operators 
$\Pc(\pv,\rv)$, $\Yc(\pv,\yv)$ and $\Rc(\pv,\yv)$, acting on $\pv$, $\yv$ and $\rv$ as above, as follows. 
For $\dv \in \Z^N$ 
and $a,v,b \in [N]$ 
we set 
\begin{align}
\Pc(\pv,\rv)(a,v,\dv) &= 
d_{v} \left(\sum_{b\in A(v)}
 \rv_{ba} ({\bf d} - \ev) 
\frac{1-\pv_{bv}({\bf d} - \eb - \ev)}
{1-\pv_{av}({\bf d} -\ea - \ev)}\right)^{-1} \text{ for } (a,v)\in \oA,\lab{F1def}\\
\Yc(\pv,\yv)(a,v,b,\dv) &= 
\frac{\pv_{av}(\dv)\big(\pv_{bv}(\dv-\ea-\ev)-\yv_{avb}(\dv-\ea-\ev)\big)}
	{1-\pv_{av}(\dv-\ea-\ev)} \text{ for } (a,v,b)\in \oA_2, \lab{FYdef}\\
\Rc(\pv,\yv)_{ab}(\dv)&= 
\frac{d_{a}}{d_{b}}\cdot \frac{1- \bad(\pv,\yv)(a,b,\dv-\eb)}{1- \bad(\pv,\yv)(b,a, {\bf d} - \ea)}
 \text{ for } a,b\in S \text{ or } a,b\in T. \lab{F2def}
\end{align}	 
 
We observed in~\cite{lw2018} that these operators are ``contractive'' in a certain sense, for particular functions $\pv$ and $\yv$, defined as follows. 
\begin{definition}\thlab{Pi-defn}
Let $\D_0\se \Z^N$ and let $\mu\in\R$. We use $\Pi_{\mu}(\D_0)$ to denote the set of pairs of functions $(\pv,\yv)$ with $\pv:\oA\times \Z^N \to \R_{\ge 0}$ and $\yv:\oA_2\times  \Z^N \to \R_{\ge 0}$ such that for all balanced  $\dv\in \D_0$, we have  
\begin{enumerate}[label={$\mathrm{(}\Pi\mathrm{\alph*}\mathrm{)}$}]
 \item\label{Pi-a}
 $0\le\pv_{av}(\dv)   \le \mu$ for all $ (a,v) \in \oA$, 
 \item\label{Pi-b}
 $\sum_{v\in \cA(a)\cap\cA(b)} \yv_{avb}(\dv)\le  \mu d_a $ for all $a\ne b\in [N]$,  and 
 \item\label{Pi-c}
$0\le \yv_{avb}(\dv)\le     \mu   \pv_{bv}(\dv)$ for all  $ (a,v,b)\in  \oA_2$.
\end{enumerate}
\end{definition}
 
For the next  lemma, we need to adapt~\cite[Lemma 5.2]{lw2018} to the present bipartite setting. 
Recall the definitions of $\cA^{\bi}$ and $\cA^{\di}$, and of $S$ and $T$ and $a'$ in Subsection~\ref{s:notation}. 
In this setting, $\cA(a)\cap\cA(b)\neq\emptyset$ if and only if both $a,b\in S$ or both $a,b\in T$. For such $a,b$ we have $|\cA(a)\sm\cA(b)| \leq 1$. Further note that $(a,v)\in\oA$ if and only if $a\in S$ and $v\in T$ in the bipartite case or $v\in T\sm\{a'\}$ in the digraph case (or $S$ and $T$ swapped), and thus $(a,v,b)\in\oA_2$ if and only if both $a, b\in S$, $a\neq b$ and $v\in T$ (bipartite) or $v\in T\sm\{a',b'\}$ (digraph); or $S$ and $T$ swapped. 

Also, for $\dv=(\sv,\tv)$ of length $\ell+n$, 
we let $Q_r^0(\dv),\, Q_r^S(\dv)\, \se\Z^{\ell+n}$ be the set of balanced and $S$-heavy, respectively, vectors of non-negative integers that have  $L_1$-distance at most $r$ from $\dv$.  
Recall that we use $1\pm \xi$ to denote a quantity between $1-\xi$ and $1+\xi$ inclusively. 
With these definitions, Lemma~5.2 in~\cite{lw2018} specialises to the following. 
 
\begin{lemma}\thlab{l:errorImplication} 
There is a constant $C>0$ such that the following holds. 
Let $\ell, n$ be integers, $N=\ell+n$, and let $\cA$ be either $\cA^{\bi}$ or $\cA^{\di}$. Let $\dv\in\Z^{\ell+n}$ such that $d_a>1$ for all $a\in [N]$. 
Let $0<\xi \le 1$ and $0<\mu_0=\mu_0(\ell,n) <C$. 
Let $(\pv,\yv)$ and $(\pv',\yv')$ be members of $\Pi_{\mu_0}(Q_2^0(\dv))$, 
and let 
 $\rv, \rv':[N]^2\times \Z^n \to \R$. Let $a,b\in S$, $v\in T$.
  \begin{enumerate}[label={(\alph*)}]
\item
If $\dv$ is $S$-heavy, $\pv_{cw}(\dv')=\pv'_{cw}(\dv')(1 \pm\xi )$ for all  $ (c,w)\in \oA$  and all $\dv' \in  Q^0_{1}(\dv)$, 
and $\yv_{cwh}(\dv')=\yv'_{cwh}(\dv')(1 \pm\xi )$ for all $(c,w,h)\in\oA_2$ and all $\dv' \in  Q^0_{1}(\dv)$, then 
$$
\Rc(\pv,\yv)_{ab}(\dv)= \Rc(\pv',\yv')_{ab}(\dv)(1+O\left(\mu_0\xi \right)). 
$$
\item
If $\dv$ is balanced, $v\neq a'$, $\pv_{cv}(\dv')=\pv'_{cv}(\dv')(1 \pm\xi )$ for all $c \in S\sm\{v'\}$ and all $\dv' \in  Q^0_{2}(\dv)$, and $\rv_{ca}(\dv')=\rv_{ca}'(\dv')(1 \pm\mu_0\xi)$ for all $c \in S\sm\{v'\}$ and all $\dv'\in  Q^1_1 (\dv)$,  then  
$$
\Pc(\pv,\rv)_{av}(\dv)=\Pc(\pv',\rv')_{av}(\dv)\left(1+O\left(\mu_0\xi \right)\right). 
$$
\item
If $\dv$ is balanced, $a\neq b$, $v\not\in\{a', b'\}$, 
$\pv_{cv}(\dv')=\pv'_{cv}(\dv')(1 \pm\mu_0\xi )$ for all $c \in S\sm\{v'\}$ and all $\dv' \in  Q^0_{2}(\dv)$, and 
$\yv_{cwh}(\dv')=\yv'_{cwh}(\dv')(1 \pm\xi )$ for all $(c,w,h)\in\oA_2$ and all $\dv' \in  Q^0_{2}(\dv)$, then 
$$
\Yc(\pv,\yv)_{avb}(\dv)=\Yc(\pv',\yv')_{avb}(\dv)\left(1+O\left(\mu_0\xi \right)\right). 
$$
\end{enumerate}
The constants implicit in $O(\cdot)$ are absolute. 
  \end{lemma}
 %

\subsection{Concentration of  random variables}
When $\dv=(\sv,\tv)$ is either the degree sequence of $\cG(\ell,n,m)$ or 
$\cB_m(\ell,n)$ then we need that $\sigma^2(\sv)$, $\sigma^2(\tv)$ and 
$\sigma(\sv,\tv)$ (in the digraph case) are concentrated (to be used in Step 4 of the template given in Subection~\ref{s:template}). 
The following is due to McDiarmid~\cite{McD}, see, e.g., Lemma~4.1 in \cite{lw2018}. 
\begin{lemma}[McDiarmid]
\thlab{l:subsetConc}  Let $c>0$ and let $f$ be a function  defined on the set of   subsets of some set $U$ such that $|f(S)-f(T)|\le c$ whenever $|S|=|T|=m$ and $|S\cap T|=m-1$. Let $S$ be a randomly chosen $m$-subset of $U$. Then for all $\alpha>0$ we have
$$
\pr\left(|f(S)-\ex f(S)| \ge \alpha c\sqrt m  \right) \le 2\exp (- 2\alpha^2).
$$
 \end{lemma}
 The following are the concentration results we need for both the sparse and the medium-dense ranges. 
\begin{lemma}
\thlab{l:sigmaConcBip} 
Let $\ell, n, m$ be integers, let $\dv=(\sv,\tv)$ be a sequence in $\cD(\G(\ell,n,m))$, $\cD(\vec \G(n,m))$, or either of the binomial models   $\cB_m(\ell,n)$ and $\vec\cB_m(n)$, and
let $\savg = m/\ell$ and $\tavg =m/n$, $\mu=\mu(\dv)=m/n(\ell-\delta^{\di})$. Let $a\in S$, $v\in T$. Then the following hold. 
\begin{enumerate}[label={(\roman*)}]

\item For 
	all $\alpha >0$ we have 
	$$\Pr\left(|s_a-\savg|\ge \alpha  \right)\le 2\exp\bigg(-\frac{\alpha^2}{2(\savg+\alpha/3)}\bigg), \ 
	\Pr\left(|t_v-\tavg|\ge \alpha  \right)\le  2\exp\bigg(-\frac{\alpha^2}{2(\tavg+\alpha/3)}\bigg).$$

\item If  $\log^3 n +\log^3 \ell =o(m)$ and $(\log n)/\sqrt n +(\log^{3/2} n)/\sqrt{m}=o(\alpha)$ 
	then $$\pr\left(|\sigma^2(\tv)- \Var\, t_v \ge \alpha \tavg +1/n\right) = o(n^{-\omega }),$$
	where $\Var\, t_v = \tavg (1-\mu)(1-1/n) (1+O(1/n\ell))$; 
	if  $\log^3 n +\log^3 \ell =o(m)$ and $(\log \ell)/\sqrt \ell +(\log^{3/2} \ell)/\sqrt{m}=o(\alpha)$ 
	then 
	$$\pr\left(|\sigma^2(\sv)- \Var\, s_a| \ge \alpha \savg +1/\ell\right) = o(\ell^{-\omega }),$$ 
	where $\Var\, s_a = \savg (1-\mu)(1-1/\ell) (1+O(1/n\ell))$. 
	Furthermore, for  $\dv=(\sv,\tv)$ in $\cD(\vec G(n,m))$ or $\vec\cB_m(n)$, 
 $$
\pr\big(|\sigma(\sv,\tv) - {\bf Covar}( s_a, t_v ) | \ge \alpha \savg +1/n\big) = o(n^{-\omega }),  
$$
 where $ {\bf Covar}( s_a, t_v )    = O( \mu)$.
\end{enumerate}
\end{lemma}

\begin{proof}   
The proofs of these concentration results are routine so we just point out the differences compared with the proof of~\cite[Lemma 6.2]{lw2018}. 
Note that $s_a$ is distributed hypergeometrically with parameters $n(\ell-\delta^{\di})$, $m$, $n-\delta^{\di}$ in all four models, and similarly, each $t_v$ is distributed hypergeometrically with parameters $n(\ell-\delta^{\di})$, $m$, $\ell-\delta^{\di}$.  
Hence, the claims in (i), and in (ii) for $\sigma^2(\sv)$ and $\sigma^2(\tv)$, follow from the proof of Lemma 6.2 in~\cite{lw2018} with only trivial adjustments. For $\sigma(\sv,\tv)$, define 
$$
g(a,b)=  \sign(a-b)  \min\{|a-b|,\sqrt x\}, 
$$
where $\sign(y)$ is $1$, $-1$ or 0 if $y$ is positive, negative or 0, respectively, 
and where $x$ is a function that satisfies $s\log n+\log^2n\ll x\ll \alpha^2sn/\log n$ as in the proof of~\cite[Lemma 6.2]{lw2018}. Let $f=\sum_{i=1}^n g(s_i,\savg)g(t_i,\tavg)$, and adapt the rest of the proof of Lemma~6.2 in~\cite{lw2018} in the obvious way. This gives the  required concentration bound for $\sigma(\sv,\tv)$
near its expected value, which is $\frac{1}{n}  \sum_{b\in S} {\bf Covar}(s_b,t_{\mate{b}}) =  {\bf Covar}(s_a, t_{\mate{a}})$. 
On the other hand,  we can bound ${\bf Covar}(s_a,t_{\mate{a}}) $  as follows. 
In $\cD(\vec G(n,m))$, the joint distribution of $(s_a,t_{a'})$ is multivariate hypergeometric, with $m$  edges chosen from $n(n-1)$ positions, and  $s_a$ and 
$t_{\mate{a}}$ are the counts for disjoint subsets of size $n-1$ each. Thus, the well known formula gives
$$
{\bf Covar}(s_a,t_{\mate{a}}) = \frac{m(n-1)^2(n(n-1)-m)}{(n(n-1))^2(n(n-1)-1)}=O(m/n^2), 
$$
which establishes the final claim for  $\cD(\vec G(n,m))$. In $\vec\cB_m(n)$ the random variables $s_a$ and $t_{v}$ are independent since we condition on $M_1(\sv) = m$ and $M_1(\tv)=m$ separately. Thus the covariance is 0. 
\end{proof}

\section{A formula for sparse digraphs and bipartite graphs} \lab{s:sparseBip}
As mentioned in the introduction, our argument is based on~\cite{lw2018}, and this section in particular has much in common with the corresponding   argument given there for the graph case.
Recall that for a given sequence $\dv=(\dvA,\dvB)$  
we define $\DeltaA=\Delta(\sv)=\max_a(s_a)$, and similarly $\DeltaB=\Delta(\tv) =\max_v(t_v)$. 

\begin{proof}[Proof of \thref{t:sparseCaseBip}]
If $\D=\emptyset$, there is nothing to prove. So fix $\dv^*\in\D$ and define 
$\hatDeltaA =2\DeltaA(\dv^*)+\ell^{\eps/6}$ and 
$\hatDeltaB =2\DeltaB(\dv^*)+n^{\eps/6}$. 
Let $\D^+$ contain all sequences $\dv\in \Z^{\ell+n}$ with $M_1(\dvA)=M_1(\dvB)=m$, $\DeltaA(\dv)\le  \hatDeltaA$ and $\DeltaB(\dv)\le \hatDeltaB$. 
   
For an integer $r\ge 0$, denote by $Q_r^0$ (or $Q_r^S$) the set of all balanced (or $S$-heavy, respectively) sequences in $\Z^{\ell+n}$ that have $L_1$ distance at most $r$ from some sequence in $\D^+$. 
(Recall that we define $\dv=(\dvA,\dvB)$ to be balanced if $M_1(\dvA)=M_1(\dvB)$ and we call it 
$S$-heavy if $M_1(\dvA)=M_1(\dvB)+1$.) 
We start by estimating the ratios of the probabilities of adjacent degree sequences in the random bipartite graph model and the random digraph model, as prescribed by Step 1 of the template in Subsection~\ref{s:template}, using the following. Recall the definition of $R_{ab}$ in~\eqref{realRatio} with $\cA$ being $\cA^{\bi}$ or $\cA^{\di}$.
\begin{claim}\thlab{RSparseBiAndDi}
Uniformly for sequences $\dv=(\dvA,\dvB)\in Q_1^S$ and for $a, b \in S$ 
$$ 
R_{ab}(\dv)=\frac{\dcA_a}{\dcA_b} \left(1 +\frac{(\dcA_a-\dcA_b) M_2 +( \dcB_{\mate{a}}-\dcB_{\mate{b}}) \delta^{\di} M_1 }{M_1^2}  \right) \left(1+O\left(\frac{\hatDeltaA^3\hatDeltaB^3}{\tavg m^2}\right)\right),
$$
where $M_1 = M_1(\dvA) -1 = M_1(\dvB)$, $M_2=M_2(\dvB)=\sum_{v\in T}\dcB_v(\dcB_v-1)$ 
 and $\delta^{\rm di}$ is 0 in the bipartite case and 1 in the digraph case.
 \end{claim}
\begin{proof} 
 Note that the claim is true when $a=b$ as $R_{aa} = 1$ by definition. Hence, we can assume $a\neq b$ in the remainder of the proof.
First, consider instead $\dv=(\dvA,\dvB)\in\D^+$, let $\ell_1$ and $n_1$ be the number of non-zero coordinates in $\dvA$ and $\dvB$, respectively. 
By definition, $\savg \ell = M_1(\dvA)\le \ell_1\DeltaA(\dv)$.
By assumption we therefore have 
\bel{deltas}
\DeltaA(\dv)\le \hatDeltaA \ll (\savg\ell)^{1 /3} 
\leq \left(\DeltaA(\dv) \ell_1\right)^{1/3},
 \ee
which readily implies that $\DeltaA(\dv) =o(\ell_1^{1/2})$. 
Similarly we find that $\DeltaB(\dv)=o(n_1^{1/2})$. 
Now apply Lemma~\ref{lem:bipRealisable}(ii), 
with $F=\emptyset$ for the bipartite graph case, and $F$ being the set of disallowed edges from $S$ to $T$ in the digraph case, 
to the sequence formed by the non-zero coordinates of $\dv$ to deduce that 
$\Num({\dv})>0$  for $\ell,n$ sufficiently large. 
We can deduce the same conclusion for all $\dv\in Q^0_{8}$, since $\DeltaA(\dv)$, $\DeltaB(\dv)$, $\ell_1$ and $n_1$ can only change by a bounded additive term 
when moving from such $\dv$ to the closest member of $\D^+$. 
Similarly, $M_1(\dvA)=\sum_i \dcA_i=\savg\ell+O(1)$ for all $\dv=(\dvA,\dvB)\in Q^0_{8}$. Thus, all such $\dv$ are $\cA$-realisable and 
$\hatDeltaA=o(M_1(\dvA)^{1/3})$ and $\hatDeltaB=o(M_1(\dvA)^{1/3})$, by~\eqn{deltas}. This and \thref{l:simpleSwitching} imply that 
\bel{Pbound}
P_{av}(\dv)=O(\hatDeltaA\hatDeltaB/M_1(\dvA)) = o(1) \quad  \mbox{for all $\dv\in Q^0_{8},\ av\in \cA$}. 
\ee 
Next consider any distinct $a,b\in S$, $v\in T$  such that $av,bv\in\cA$ (i.e.~$(a,v,b)\in\oA_2$) and $\dv\in Q_6^0$ with $d_a>0$ and $d_v>0$. 
Then $\dv-\ea-\ev \in Q_{8}^0$ and hence $\Num(\dv-\ea-\ev)>0$ from above, 
and also $P_{av}(\dv-\ea-\ev)=O(\hatDeltaA\hatDeltaB/\savg \ell) <1$ using~\eqref{Pbound}. 
Thus, for $\ell,n$ sufficiently large, $\Num_{av}(\dv)>0$ by \thref{trick17}, 
and we have $\Num_{av}(\dv)<\Num(\dv)$ since $P_{av}(\dv)<1$ for similar reasons. This establishes the hypotheses for $\Num_{av}(\dv)$ and $\Num(\dv)$ in \thref{l:recurse} whenever they are needed below. 

It now follows that $Y_{avb}(\dv)=O\left((\hatDeltaA\hatDeltaB/M_1(\dvA))^2\right)$ for $\dv\in Q_6^0$; 
if $d_a$ or $d_v$ is 0 then this is immediate, and otherwise it follows from~\eqn{FYdef}   in view of~\eqref{Pbound}, and noting that the numerator is non-negative by definition. 
Next, definition~\eqn{eq:bad} yields 
$\Bad(a,b,\dv)=O\left((\hatDeltaA\hatDeltaB)^2/\tavg M_1(\dvA)\right)$ 
for all distinct $a,b\in S$ and all $\dv\in Q_{6}^0$ with $d_a>0$. 
(In the current setting 
$|\cA(a)\setminus \cA(b)|$ is 0 and $\cA(a)\cap \cA(b) = T$ for the bipartite case, 
and $\cA(a)\setminus \cA(b) =\{\mate{b}\}$ and $\cA(a)\cap \cA(b) = T\sm\{\mate{a},\mate{b}\}$ for the digraph case.) 
Thus~\eqref{Ratformula} gives $R_{ab}=d_a/d_b(1+O((\hatDeltaA\hatDeltaB)^2/\tavg M_1(\dvA)))$ for all $\dv\in Q_{5}^S$ and all distinct $a,b\in S$ such that $d_a, d_b>0$. 
Now let $\dv\in Q_4^0$ and $v\in T$ with $d_v>0$. 
We want to evaluate $\sum_{b\in\cA^*(v)}d_b$ in \thref{l:recurse}(a) and 
recall that $\cA^*(v)$ is the set of vertices $b$ such that $bv\in \cA$ and $\Num_{bv}(\dv)>0$. 
If $d_b>0$ then $\Num_{bv}(\dv)>0$ as noted above. 
Therefore 
$\sum_{b\in\cA^*(v)}d_b = \sum_{b\in\cA(v)}d_b  = M_1(\dvA)- \delta^{\di}d_{\mate{v}}$ for such $\dv$. Thus \thref{l:recurse} gives 
\bel{Pavsparse}
P_{av}(\dv) =d_ad_v/M_1(\dvA)
	\left(1+O\left((\hatDeltaA\hatDeltaB)^2/\tavg M_1(\dvA)\right)\right)
\ee
for all $\dv\in Q_4^0$, $av\in \cA$ (if $d_a$ and $d_v$ are non-zero, the proposition applies as mentioned above, and if either is 0 then the claim holds trivially). 
Using a similar argument,~\eqn{FYdef} gives 
$$
Y_{avb}(\dv)=\frac{d_a[d_v]_2 d_b}{M_1(\dvA)^2}\left(1+O\left(\frac{(\hatDeltaA\hatDeltaB)^2}{\tavg M_1(\dvA)}\right)\right)
$$
for all $\dv\in Q^0_2$, all $(a,v,b)\in \oA_2$ with $a\in S$,
and where $[x]_2$ denotes the falling factorial $x(x-1)$. 
Next, apply these results to the definition~\eqn{eq:bad} of $\Bad$ for  $\dv\in Q^0_2$ and distinct $a,b\in S$, 
and note that $\cA(a)\setminus \cA(b) = \{b'\}$ for the digraph case and is empty in the bipartite case. 
This  gives the sharper estimate 
$$\Bad(a,b,\dv) =  
		\left(\frac{d_b M_2(\dvB)}{M_1(\dvA)^2} 
		+\delta^{\di} \frac{d_{\mate{b}}}{M_1(\dvA)} 
		\right)
		\left(1+O\left(\frac{\hatDeltaA^2\hatDeltaB^2}{\tavg M_1(\dvA)}\right) \right)
		+O\left(\frac{\hatDeltaB^2 \hatDeltaA}{M_1(\dvA)^2}\right)
$$ 
for  $\dv\in Q^0_2$, which we note is $O(\hatDeltaA\hatDeltaB/M_1(\dvA))$ as 
$M_2(\dvB) \leq \hatDeltaB M_1(\dvB)=\hatDeltaB M_1(\dvA)$. Recalling that $t\le \hatDeltaB(\dv)+2$ let us write $\hatDeltaB^2 \hatDeltaA/M_1(\dvA)^2 =O(\hatDeltaB^3 \hatDeltaA/tM_1(\dvA)^2)$. 
Thus, for all $\dv\in Q_1^S $ and all distinct $a,b\in S$, and noting that $M_2=M_2(\dvB)$ changes by a negligible additive term $O(\hatDeltaB)$ under bounded perturbations of the elements of the sequence $\dv$,
\bel{Req1}
R_{ab}(\dv) =\frac{d_a}{d_b}\frac{\big(1  -  ( d_b-1)M_2/M_1^2 - \delta^{\di}(d_{\mate{b}})/M_1 \big)}
{\big(1 -  ( d_a-1)M_2/M_1^2 -   \delta^{\di}(d_{\mate{a}})/M_1\big)}
\bigg(1+O\bigg( \frac{(\hatDeltaA\hatDeltaB)^3}{tM_1^2}\bigg)\bigg),
\ee
by \thref{l:recurse}(b), which implies the claim since $d_a=\dcA_a$ and $d_{\mate{a}}=\dcB_{\mate{a}}$, and similarly for $d_b$ and $d_{\mate{b}}$. \end{proof}

We  next  make the definitions of probability spaces necessary to apply Lemma~\ref{l:lemmaX} (see Steps 2 and 3 in the template). 
Let $\Omega$ be the underlying set of $\cB_m(\ell,n)$ in the bipartite case and of $\vec\cB_m(n)$ in the digraph case. 
Let 
$$
\W=\left\{ \dv\in \D^+ : \max\{ \sigma^2(\dvA)\ell,\sigma^2(\dvB)n \} \le 2m\right\}
$$ 
 in the bipartite case, and the same but with the additional restriction that 
$$
|\sigma(\dvA,\dvB)|\le 2 s 
$$
in the digraph case. 
Let $ \cS'=\cD(\G(\ell,n,m))$ in the bipartite case, and $ \cS'=\cD(\vec \G(n,m))$ in the digraph case. 
We now turn to define a second probability space $\cS$ on the underlying set $\Omega$ (see Step 3 of the template). 
Recall the definition of $\correct(\dv)$ in~\eqn{corrH}
and let $\apx(\dv) 
=\pr_{\cB_m}(\dv) \correct(\dv)$, where $\cB_m$ is $\cB_m(\ell,n)$ in the bipartite case and $\vec\cB_m(n)$ in the digraph case. 
Now define the probability function in $\cS$ by 
\bel{PS}
\pr_\cS(\dv) = \apx(\dv)/\sum_{ \dv'\in\W}\apx(\dv') = \frac{ \pr_{\cB_m}(\dv) \correct(\dv)}{\ex_{\cB_m} (\mathbbm{1}_{\W}\correct)}
\ee
for $\dv\in \W$, and  $\pr_{\cS}(\dv)=0$ otherwise. 
The graph $G$ has vertex set $\W$ where two vertices are adjacent if they are either of the form $\dv-\ea$, $\dv-\eb$ for some $a$ and $b$ in $S$, or  $\dv-\ev$, $\dv-\ew$ for some $v$ and $w$ in $T$.

We need to estimate the probability of $\W$ in the two probability spaces $\cS$ and $\cS'$ 
(see Step 4 in the template). For convenience, we simultaneously make a similar estimate of   $\pr_{\cB_m}(\W)$ for use outside the present proof. Note that $\pr_{\cS}(\W)=1$ by definition. We combine estimating $\pr_{\cS'}(\W)$ and estimating the expressions $\correct (\dv)$ and $\ex_{\cB_m} (\mathbbm{1}_{\W}\correct)$ in \eqn{PS} for later use. In the following, let $\dv$ be in either of $\cS'$ or $\cB_m$. 
We first claim that $\dv\in \D^+$ with high probability. 
Clearly, $M_1(\dvA)=M_1(\dvB)$ for all such $\dv$ since this is true for any bipartite graph (or digraph) with sequence $(\dvA,\dvB)$ by the definition of $\cB_m$. 
Letting  $\DeltaA^*=\DeltaA(\dv^*)$ and $\DeltaB^*=\DeltaB(\dv^*)$, we have by definition $\hatDeltaA\ge \DeltaA^*+\ell^{\eps/12}\sqrt {\DeltaA^*}\ge \savg+\ell^{\eps/12}\sqrt {\DeltaA^*}$ and similarly $\hatDeltaB\ge \tavg+n^{\eps/12}\sqrt {\DeltaB^*}$. 
Thus, for   $\dv\in  \Omega$, in either $\cS'$ or $\cB_m$,
\begin{align}\lab{aux885}
\pr (\dcA_a > \hatDeltaA)&\le \pr\big(\dcA_a>\savg+ \ell^{\eps/12 }\sqrt {\DeltaA^*}\big) 
=o(\ell^{- \omega}), \text{ and}\\
\pr(\dcB_v > \hatDeltaB)&\le  \pr\big(\dcB_v>\tavg+ n^{\eps/12 }\sqrt {\DeltaB^*}\big) 
=o(n^{- \omega})
\end{align} 
by \thref{l:sigmaConcBip}(i) and noting that $\DeltaA^*,\DeltaB^*\to\infty$. The union bound now gives that, with probability $1- o(n^{- \omega}+\ell^{- \omega})$, $\DeltaA(\dv)\le\hatDeltaA$  and $\DeltaB(\dv)\le\hatDeltaB$ and hence $\dv\in \D^+$, in both $\cS'$ and $\cB_m$. 
We next argue that the $\sigma$-terms are concentrated for $\dv$ chosen in $\cS'$ or $\cB_m$. Let $\alpha = \max\{\log^4\ell/\sqrt{\ell},\log^4n/\sqrt{n}\}$ and note that the conditions  $\log^3n+\log^3\ell=o(m)$, $(\log n)/\sqrt{n}+(\log^{3/2}n)/\sqrt{m}=o(\alpha)$,  and $(\log \ell)/\sqrt{\ell}+(\log^{3/2}\ell)/\sqrt{m}=o(\alpha)$ in \thref{l:sigmaConcBip}(ii) follow from $n/\log^4n + \ell/\log^4\ell=o(m).$ 
Recalling that $\mu=m/n(\ell-\delta^{\di})$ we thus deduce that $\sigma^2(\dvB) = \tavg (1-\mu) \big(1+O(\alpha)\big)$ with probability $1-o(n^{-\omega})$, that $\sigma^2(\dvA)=\savg (1-\mu) \big(1+O(\alpha)\big)$ with probability $1-o(\ell^{-\omega})$, and, in the digraph case, $\sigma(\dvA,\dvB)=O(\mu+s\alpha)$ with probability $1-o(n^{-\omega})$, by \thref{l:sigmaConcBip}(ii).   
This already implies that $\Pr_{\cS'}(\W) =  1-o(n^{-\omega}+ \ell^{- \omega})$, in preparation for applying  Lemma~\ref{l:lemmaX}, and similarly $\Pr_{B_m}(\W) =  1-o(n^{-\omega})$. 

Now note that if $\sigma^2(\dvB)= t (1-\mu)\big(1+O(\alpha)\big)$, 
then the term  $\sigma^2(\dvB)/t(1-\mu)$ in the exponent of $\correct(\dv)$ 
is $1+O(\alpha)$. Similarly for the term  $\sigma^2(\dvA)/s(1-\mu)$. 
Furthermore, the term $\sigma(\dvA,\dvB)/s(1-\mu)$ in the digraph case  is $O(\alpha)$. 
It follows using the strong concentration shown in the previous paragraph that 
\bel{HConc}
\correct(\dv)=1+O(\alpha) \text{ with probability } 1-\bar n^{-\omega} \text{ for $\dv \in \cB_m$}, 
\ee 
where $\bar n=\min\{n,\ell\}$. 
Recall that $\sigma^2(\dvA)\le 2s$, $\sigma^2(\dvB)\le 2t$ and, in the digraph case, $|\sigma(\dvA,\dvB)|\le 2s$ for all $\dv\in \W$. Thus, $\correct(\dv)=\Theta(1)$ for $\dv\in\W$,  using the fact that $\mu < 1/2$, say. 
This and \eqn{HConc} then imply that $\ex_{\cB_m} (\mathbbm{1}_{\W}\correct) = 1+O(\alpha)$. 

To apply \thref{l:lemmaX} (see Step 5 in Section~\ref{s:template}), the final condition we need to show is that the ratios of probabilities satisfy 
\bel{eq:target1}
\frac{\Pr_{\cS'}(\dv-\ea)}{\Pr_{\cS'}(\dv-\eb)} =e^{O(\delta)}\frac{\Pr_{\cS}(\dv-\ea)}{\Pr_{\cS}(\dv-\eb)}
\ee
whenever $\dv-\ea$ and $\dv-\eb$ are elements of $\W$ that are adjacent in the auxiliary graph $G$ defined above, for $\delta=\delta(\hatDeltaA,\hatDeltaB)$ independent of $\dv$ which we specify below, where  the constant implicit in $O()$ is independent of $\dv$ and $\dv^*$. 
Compare the ratio formula 
\begin{align}\label{eq:ratioSparse}
\frac{\Pr_{\cS}(\dv-\ea)}{\Pr_{\cS}(\dv-\eb)} 
= \frac{\apx(\dv-\ea)}{\apx(\dv-\eb)}
\end{align}
in~\eqref{H} with the expression in \thref{RSparseBiAndDi} 
when $a$ and $b$ are in $S$. Then use the identity $n\sigma^2(\tv)=M_2-(t-1)M_1$ and recall that $M_1=\mu'n\ell$   in~\eqref{H} to deduce that 
 \bel{ratios2}
 R_{ab}(\dv) = \frac{\apx(\dvA-\ea,  \dvB)}{\apx(\dvA-\eb, \dvB)} \bigg(1+O\bigg( \frac{\hatDeltaA^3\hatDeltaB^3}{\tavg ^3 n^2}  \bigg)\bigg) 
\ee
whenever $a,b \in S$ and 
$\dv\in Q_1^S$, 
where we use   
$(n+\delta^{\di}-\dcA_b)/(n+\delta^{\di}-\dcA_a)=\exp\big((\dcA_a-\dcA_b)/n+O(\DeltaA^2/n^2)\big)$,  $1/(1-\mu)=1+O(\mu)$, 
$M_2(\dvB)=O(\DeltaB M_1)$ 
and $(\DeltaA\DeltaB)^2/M_1^2\le (\DeltaA\DeltaB)^3/tm^2$, $\DeltaA\le \hatDeltaA$, $\DeltaB\le \hatDeltaB$, and the most significant error term derives from \thref{RSparseBiAndDi}. 
The corresponding statements when $a,b \in T$ follow accordingly (after swapping 
$S\leftrightarrow T$, $\savg\leftrightarrow \tavg$ and $\ell \leftrightarrow n$ in the conclusion of the argument). 
Equation~\eqn{ratios2} now implies that  
$$
\frac{  \pr_{\cS'}(\dv-\ea)}{ \pr_{\cS'}(\dv-\eb) } = R_{ab}(\dv) = e^{O(\delta)}\frac{  \pr_{\cS}(\dv-\ea)}{ \pr_{\cS}(\dv-\eb)}
$$
whenever $\dv-\ea$ and $\dv-\eb$ are adjacent elements of $\W$, 
where we may take 
$ \delta= (\hatDeltaA\hatDeltaB)^3(1/\tavg m^2 + 1/\savg m^2)$ which 
is the error term in~\eqn{ratios2} together with the symmetric error after the swap. 
It is clear that the diameter $r$ of $G$ is at most $M_1 (\dvA) +M_1(\dvB)= 2m$.
Lemma~\ref{l:lemmaX}  then   implies  that $P_{\cS'}(\dv)=e^{O( r\delta+\eps_0)}P_{\cS}(\dv)$ for $\dv\in \W$, where $\eps_0 = 1/n+1/\ell$. 
To proceed from here,  since we found that $\ex_{\cB_m} (\mathbbm{1}_{\W}\correct) = 1+O(\alpha)$, equation~\eqn{PS} implies $P_{\cS}(\dv) = \apx(\dv)(1+O(\alpha))$ for $\dv\in \W$. Hence, 
\bel{generalError}
P_{\cS'}(\dv^*)=e^{O(r\delta+\eps_0+\alpha)}  \apx(\dv^*). 
\ee
Note that $\alpha=O(n^{\eps-1/2}+\ell^{\eps-1/2})$, and  
$$r\delta+\eps_0=  O\left((\hatDeltaA\hatDeltaB)^3\left(\frac{1}{tm}+\frac{1}{sm}\right)\right)
	=  O\left( \DeltaA(\dv^*)^3\DeltaB(\dv^*)^3 (\ell n)^{\eps/2}\left(\frac{1}{tm}+\frac{1}{sm}\right)\right).$$
The theorem for $\dv\in \W$ follows since $\cS'$ is $\cD(\G(\ell,n,m))$ or $\cD(\vec\G(n,m))$, respectively, and by definition of $\apx$. 
On the other hand, to treat the elements of $\D^+\setminus \W$, where some $\sigma$-term may be unbounded, we still have~\eqn{ratios2} holding for all  $\dv\in Q_1^S$.  
For any  $\dv\in \D^+$, there is  a telescoping product of such ratios starting with $\dv'\in \W$, of length at most $2m$, which shows that $\pr_{\cS}(\dv)/\apx(\dv)=  \pr_{\cS}(\dv')/\apx(\dv')(1+O(m\delta))$.   From this,  the required formula for  $\pr_{\cS}(\dv)$ follows for all $\dv\in\D$, in both the bipartite graph and digraph cases. 
\end{proof}

 
\section{Proof of \thref{t:mainbip,t:bipmodel,t:edgeprobability}} \lab{s:denseBip}

\def\epsV{\varepsilon} 
\def\epsA{\varepsilon}  
\def\sigmaA{\sigma_S}
\def\sigmaB{\sigma_T}

In this section we prove Theorem~\ref{t:mainbip}. 
Following the template in Subsection~\ref{s:template} we first consider the ratio of probabilities of ``adjacent'' degree sequences. 

To estimate those ratios we first present functions that approximate the ratios and the probabilities. We  write these approximations parameterised to facilitate identifying negligible terms. 
We express the formulae for the approximations of $P$, $Y$ and $R$ in both the bipartite graph case and the digraph case simultaneously, again using $\delta^{\di}$ as the indicator variable which is 1 in the digraph case and 0 in the bipartite case. 
For integers $\ell$ and $n$, and a sequence of real numbers $\mu$, $\epsA_a$, etc.,
we define the expressions   
\remove{\ac{There was a discussion here to change to a different parametrisation, using $\mu/t^2$ instead of $1/t\ell$. In the end I decided against it. There are two reasons: For once, this introduces a slight error for the digraphs, and we'd have to argue once again that this is negligible. Second, in the old version the $\sigma_T^2$ term always comes with a factor $1/t \ell$. So in maple, we use a variable $sV$ that stands for $\sigma_T^2/t\ell$. This is important, as $\sigma$ by itself doesn't necessarily go to 0. With this extra factor, saying that the square of this term is negligible is much easier. I hope you don't mind. I thought about for a while when tackling the maple. \\
I DID change: $1/n-1$ to $\mu/s$ in both $\pi$ and $\rho$. This way, the only place where $n$ and $\ell$ appear is together with the $\sigma^2$ terms. That's less to worry about in the worksheet. }}
\begin{eqnarray*}
\pi & = &\mu (1+\epsA_a)(1+\epsV_v)
		\left(1 - \frac{\mu\epsA_a\epsV_v- \epsA_a \sigmaB^2/\tavg \ell  -\epsV_v \sigmaA^2  /\savg n}{1-\mu}
		+ \frac{\delta^{\mathrm{di}}(\epsV_{a'}+\epsA_{v'})\mu}{s}\right),\\
\rho  
	& = &\frac{1+\epsA_a}{1+\epsA_b}\cdot\frac{1-  \mu(1 + \epsA_b)+\mu/s}
	{1- \mu(1+ \epsA_a)+ \mu/s} \\
	&& \times \left(1 
	+\frac{\epsA_a-\epsA_b}
	{(1-\mu)}\left(\frac{\sigmaB^2 }{(1-\mu)t\ell} 
	- \frac{1}{\ell}
	\right)
	+\frac{\delta^{\mathrm{di}}(\epsV_{a'}-\epsV_{b'})\mu}{s(1-\mu)}
	\right).
\end{eqnarray*}

In the calculations below, there are small changes in most of the variables that turn out to have a negligible effect. Changes in the various occurrences of $\eps$-type terms, however, need to be tracked precisely. In particular, we use 
\begin{itemize}
\item 
$\pi(x,y)$ to stand for $\pi$ with  $\epsA_a,\epsV_v$ replaced by $x,y$, 
\item
$\pi(x,y,z,w)$ to stand for $\pi$ with  $\epsA_a,\epsV_v,\epsV_{a'},\epsA_{v'}$ replaced by $x,y,z,w$,
\item 
$\rho(x,y,z,w)$ to stand for $\rho$ with  $\epsA_a,\epsA_b,\epsV_{a'},\epsV_{b'}$ replaced by $x,y,z,w$. 
\end{itemize} 
Recall that we consider  sequences $\dv = (\sv,\tv)$, where $\sv$ and $\tv$ have length $\ell$ and $n$, respectively, and where $\ell = n$ in the digraph case. 
Also $\mu =  \frac12 M_1(\dv) /  |\cA|$, where $|\cA|=(1-\delta^{\mathrm{di}})n\ell+\delta^{\mathrm{di}}n(n-1)$, 
$\savg=(\sum_{a\in S}s_a)/\ell$ and $\tavg=(\sum_{v\in T}t_v)/n$. 
Noting that $S\cap T=\emptyset$, we set $\eps_a=(s_a-\savg)/\savg$ for $a\in S$, $\eps_v=(t_v-\tavg)/\tavg$ for $v\in T$,  
$\sigmaA^2 = \sigma^2(\sv)$, and $\sigmaB^2=\sigma^2(\tv)$.
Then with $*$ standing  for either of $\mathrm{bi}$ or $\mathrm{di}$, we define
$$
P^*=\pi, \quad R^* = \rho,$$ 
$$
Y^*= \pi\cdot \pi(\epsA_b, \epsV_v-1/t, \epsV_{b'},\epsA_{v'})\cdot
\left( 
1+\frac{\mu(1 + \epsA_a)-  \mu^2(1+\epsA_a+\epsA_b)  }{t(1-\mu)}
 \right) ,
$$
so that $P_{av}^{\bi}$ etc.\ are functions of degree sequences. 
For example, the  edge probability function for bipartite graphs with ``classical'' parameters $d_a$ etc is given by 
\begin{eqnarray*} 
P_{av}^{\bi} & = &  \frac{d_ad_v}{m} \left(
	1-\frac{n\ell(d_a-s)(d_v-t)}{m(n\ell-m)}-\frac{(d_a-s)\sigma^2(\tv)n}{ts(n\ell-m)}
	-\frac{(d_v-t)\sigma^2(\sv)\ell}{ts(n\ell-m)}\right)
\end{eqnarray*} 
for $a\in S$, $v\in T$. 
One can compute similar expressions for $P_{av}^{\di}$,  the ratio functions $R_{ab}^{\bi}$ and $R_{ab}^{\di}$, and the 2-path probabilities $Y_{avb}^{\bi}$ and $Y_{avb}^{\di}$. 
 In particular, we remark that $R^*_{ab}(\dv)$ written with ``classical'' parameters is just the expression in~\eqref{H}.
The functions $P^*$,  $Y^*$ and $R^*$ are our ``guessed'' probability and ratio functions and we will show that they approximate the actual  functions sufficiently well. 
The following implies that they are close to fixed points of the operators defined in~(\ref{F1def}--\ref{F2def}). 
 
\begin{lemma}\thlab{l:mapleHigher}
Let $n, \ell$ be integers and let $1/2\le \degspread < 3/5$. 
Let   
$\cA$ be as in either the bipartite or the digraph case,  
let $\dv=(\sv,\tv)$ be a sequence of length $\ell+n$, 
 let $\mu=M_1(\dv)/2|\cA|$, and assume that $\mu<1/4$. 
Furthermore, let $\savg$ and $\tavg$ be the average of $\sv$ and $\tv$, respectively, and  $\eps= \bar d^{\degspread-1}$, where $\bar d = \min\{ \savg, \tavg\}$, 
and assume that 
	$\max_{a\in S} |s_a - \savg|/\savg \leq \eps$, 
	$\max_{v\in T} |t_v - \tavg|/\tavg \leq \eps$. 
Let $*$ stand  for either of $\bi$ or $\di$.   
Then 
\begin{itemize}
\item[(a)] $\Rc(P^*,Y^*) _{ab}( \dv)=  R^*_{ab} ( \dv)
 \left(1 + O\left(\mu\eps^4\right)\right)$ for all $a,b\in S$, 
\item[(b)] 
$\Pc(P^*,R^*)_{av}(\dv) =  P^*_{av}(\dv)
 \left(1 + O(\mu\eps^4)\right)$  for all $a\in S$, 
 $av\in \cA $
\item[(c)] 
$\Yc(P^*,Y^*)_{avb}(\dv) =  Y^*_{avb}(\dv)
 \left(1 + O(\mu\eps^4)\right)$  for all distinct $a,b\in S$, 
 $v\in T $, $av,vb\in\cA$.
\end{itemize}
\end{lemma}

\begin{proof} 
This follows the proof of~\cite[Lemma 7.1]{lw2018} very closely, with modifications due to the bipartite setting.

We first present some convenient approximations of $P^*$ and $\pi$ for use when their parameters have been slightly altered. Let $\dv'=(\sv',\tv')$ be a sequence  where $\sv'$ and $\tv'$ are of length $\ell$ and $n$, respectively, such that $\dv'$  is at $L_1$ distance $O(1)$ from $\dv$, and with $d_a'=d_a-j_a$ and $d_v'=d_v-j_v$ for some $a\in S$, $v\in T$. 
Here and in the following, the bare symbols $\mu$,  $\epsA_a$ and so on are defined with respect to the original sequence $\dv$, whilst $\mu'$,  $\epsA_a'$, etc., are defined with respect to $\dv'$. For such a sequence $\dv'$ we have that $\mu(\dv')$ is $\mu'= \mu +O(1/ n\ell)$ by definition of $\mu(\dv')$.    
Therefore, the variables $\epsA_a$ and $\epsV_v$ change to 
 $\epsA_a'=\epsA_a-j_a/\savg+O(1/\mu n\ell)$ and $\epsV_v'=\epsV_v-j_v/\tavg+O(1/\mu n\ell)$. (Note that this takes into account that 
$\epsA_a'$ and $\epsV_v'$   are defined with respect to 
$s(\dv')$ and $t(\dv')$.) 
Furthermore, 
$$\sigma^2 (\sv')-\sigma^2 (\sv)=O(( \max_{c\in S} |s_c-\savg| +1 )/ n )=O(\eps \savg/n),$$ 
and similarly, for $\sigma^2 (\tv')$. Hence, by definition of $P^*=\pi$ and the preceding considerations,
\bel{piwithprime2}
P^*_{av}({\dv'})
=\pi( \epsA_a-j_a/\savg,\epsV_v-j_v/\tavg)\big(1+O( 1/\mu n\ell)\big),
\ee
That is, the changes from $\mu$, $s$, $t$, $\sigmaA^2$ and $\sigmaB^2$ to 
$\mu'$, $s'$, $t'$, $(\sigmaA^2)'$ and $(\sigmaB^2)'$ are negligible in the formula for $P^*$.

For (a), we note first that $\Rc( \Pgr,\Ygr)_{aa}(\dv) = 1 = R^*_{aa}(\dv)$ by definition of $\rho$ and $\Rc$ in~\eqref{F2def}. Assume now that $a\neq b$. Using~\eqn{F2def} to evaluate $\Rc( \Pgr,\Ygr)_{ab}(\dv)$, we estimate the expression 
$\bad(a,b,\dv-\eb)=\bad(\Pgr,\Ygr)(a,b,\dv-\eb)$ for which, in turn, we need to estimate 
$\sum Y^*_{avb}(\dv-\eb)$, where the sum is over all $v\in T$  such that both $av$ and $bv$ are allowable (see \eqref{def:bad}). 
By definition and~\eqref{piwithprime2}, 
$$
Y^*_{avb}(\dv-\eb) = \pi \cdot  \pi(\epsA_{b}- 1/\savg,\, \epsV_{v}- 1/\tavg)
	\cdot\left(1+\frac{\mu(1+\epsA_a) - \mu^2(1+\epsA_a+\epsA_b)}{t(1-\mu )} +O( 1/\mu n\ell) \right),
	$$
where we use $\epsA_a$, $\epsA_b$ and $\mu $ in the third factor (rather than the altered versions $\epsA_a'$ etc.)~using the same reasoning as for obtaining~\eqref{piwithprime2}.
Consider expanding this expression for $Y^*_{avb}(\dv-\eb)$ ignoring terms of  order  $\eps^4$, and hence also ignoring those of order $s^{-2}$, $t^{-2}$, $\eps^2/t$, and $\eps^2/s$ since $\eps^2\ge 1/\bar d =\max\{1/t,1/s\}$. 
A convenient way to do this is to make substitutions $\epsA_v=y_1\epsA_v$,   $1/t  = y_1^2/t $,  
$\mu  = y_2 \mu$, $1/n =y_1^2y_2/n$,   
and so on where $y_1$ represents a parameter of size $O(\eps)$ and $y_2$ one of size $O(\mu)$ (for instance, $\sigmaB^2/t\ell$ is $O(\eps^2\mu)$), and then expand about  $y_1=0$ and drop terms of order $y_1^4$.
Since $Y^*_{avb}(\dv-\eb)$ gains a factor  $y_2^2$ via the factors of $\mu$ in $\pi$,  each term containing $y_1^i$ is of order  $y_1^iy_2^2$ and is hence $O(\mu^2\eps^i)$. 
Next, removing the `sizing' variables $y_i$   by setting them equal to 1, and then  expanding the result  about  $(\epsA_{v'},\epsV_v) =(0,0)$  and retaining all terms of total degree at most 3 in $ \epsA_{v'}$ and $ \epsV_v$, we get
 \begin{align*}
Y^*_{avb}(\dv-\eb)
&= c_{0} 
		+ c_{10}\epsV_{v}
		+ c_{01}\epsA_{v'}
		+ c_{20}\epsV_{v}^2
 		+O(\mu^2\eps^4),
\end{align*}
where the functions  $c_{0}$, $c_{10}$, $c_{01}$, and $c_{20}$   
are independent of $\epsV_v$ and $\epsA_{v'}$, with $c_0$, $c_{10}$ and $c_{01}$ linear in $1/\savg$ and $  1/\tavg$, and $c_{20}$ constant in those variables. (By calculation, the third order terms all turn out to be absorbed by the error term.  Furthermore, the relative error $1/\mu \ell n$ in the previous expression for $Y^*$ yields an absolute error $O(\mu/\ell n)=O(\mu^2\eps^4)$ since~$Y^*$ is~$O(\mu^2)$.)  We note that $c_{01}$ has a factor $\delta^{\di}$ since $\epsA_{v'}$ in $\pi$ has such a factor. 
  
Then considering the definition of $\bad(a,b,\dv-\eb)$ in~\eqn{def:bad} we find that the second summation can be written as 
\begin{align*}
\Sigma _{\bad} 
&= \sum_{v\in \cA(a)\cap \cA(b)}Y^*_{avb}(\dv-\eb)\\
&= \sum_{v\in \cA(a)\cap \cA(b)}
		(c_{0} 
		+ c_{10}\epsV_{v}
		+ c_{01}\epsA_{v'}
		+ c_{20}\epsV_{v}^2
		+ O(\mu^2\eps^4)) \\
&= n c_0 +  nc_{20}\sigmaB^2/\tavg^2\\
&\qquad
	+\delta^{\di}
	\left(
	-2c_0 - c_{10}(\epsV_{a'}+\epsV_{b'})
	- c_{01}(\epsA_a+\epsA_b)
	-c_{20}(\epsV_{a'}^2+\epsV_{b'}^2)
	\right) +O(n \mu^2\eps^4),  
\end{align*}
since $a$ and $b$ are distinct elements of $S$ (in which case $\cA(a)\cap\cA(b)$ is $T$ in the bipartite case and is $T\sm\{a',b'\}$ in the digraph case),
and where we also use that $\sum_{v\in T}\epsV_v =0$ and that, in the digraph case, $\sum_{v\in T}\epsA_{v'} = \sum_{a\in S}\epsA_a =0$.  
Noting that   $\cA(a)\setminus \cA(b)$  is $\emptyset$ in the bipartite case 
and consists of just $\mate{b}$ in $T$ in the digraph case,  
we can write $\bad(a,b,\dv-\eb)$ in~\eqn{F2def}, by using~\eqn{def:bad}, as 
\begin{align*}
\bad(a,b,\dv-\eb) &= \frac{1}{d_a} 
	\left(\Sigma_{\bad}
	+\delta^{\di} \pi(\epsA_a,\epsV_{b'},\epsV_{a'},\epsA_b-1/\tavg) +O(1/ \ell n) 
\right), 
\end{align*}
where $d_a =s_a = (1+\epsA_a)\savg$,  and  the $O(1/n\ell)$ term captures the fact that we use $\mu=\mu(\dv)$, $\sigma^2(\sv)$, and $\sigma^2(\tv)$, respectively, instead of $\mu(\dv')$, $\sigma^2(\sv')$, and $\sigma^2(\tv')$, in the formula for $\pi$ applied to an altered sequence $\dv'$.  
The error term $O(1/\ell n)$ here, together with the one from $\Sigma_{\bad}$ above, produce an additive error in $\bad(a,b,\dv-\eb) $ of $O(\mu\eps^4)$ since $n/d_a=n/s_a\sim 1/\mu$ and $1/\ell n<\mu\eps^4$.
Substituting the above expression, stripped of its error  terms, into
\begin{align*}
 \frac{\Rc (P^*,Y^*)_{ab} (\dv)}{\rho} -1
 &= \frac{1}{\rho}\cdot \frac{(1+\epsA_a)(1-\bad(a,b,\dv-\eb))}{(1+\epsA_b)(1-\bad(b,a,\dv-\ea))}  -1
\end{align*}
and simplifying gives a rational function $\widehat F$ which satisfies $\Rc(P^*,Y^*)_{ab}(\dv)/\rho  = 1+ \widehat F + O(\mu\eps^4)$.  After inserting the  size variables $y_1$ and $y_2$  into $\widehat F$ as specified above (and here it may be convenient to use $n=\delta^{\di}+s/\mu$), and simplifying, we find that $\widehat F$ has $y_2$ as a factor (of multiplicity 1), and its denominator is nonzero at $y_1=0$.  Then expanding the expression in powers of $y_1$ shows that  $\widehat F=O( y_1^4)$. Along with the extra factor $y_2$, this implies $\widehat F=O(\mu\eps^4)$. Part (a) follows. 

To prove part (b) let $av\in\cA$ with $a\in S$ and let $\dv'$ be the sequence $\dv -\ev$. 
Note that, analogous to~\eqn{piwithprime2}, 
the differences in the values of $\mu$, $s$, $t$ in $\rho$ between $\dv$ and $\dv'$ are negligible, as are the differences in $\epsA_a$ and $\epsA_b$ 
for $b\in \cA(v) = S\sm\{v'\}$ (note that with these assumptions $v$ is never equal to $a, a', b$ or $b'$; so the values of $\epsA_a$, $\epsA_b$, $\epsA_{a'}$, $\epsA_{b'}$ only change since $\savg$ changes).  Hence, we also have  
$$ R^*_{ab}(\dv')  
=\rho\cdot \left(1 +O\left( 1/ \mu  n\ell  \right)\right)$$
for $a,b\in S$, $v\in \cA(a)\cap \cA(b)$.
Therefore, by definition \eqref{F1def},  
\begin{align}\lab{aux1552}
\Pc( P^*, R^* )_{av}(\dv)  
	&= d_{v} \left(\sum_{b\in  \cA(v)}
		R^*_{ba} (\dv-\ev)
		\frac{1- P^*_{bv}(\dv  - \eb - \ev)}
		{1-P^*_{av}( \dv -\ea - \ev)}\right)^{-1}\nonumber\\
	&= d_{v} \left(\sum_{b\in \cA(v)} 
		\rho (\epsA_b,\epsA_a,\epsV_{b'},\epsV_{a'} ) \cdot 
		\frac{1-\pi\left(\epsA_b- 1/\savg ,\epsV_v- 1/\tavg ,\epsV_{b'},\epsA_{v'}\right)}
		{1-\pi\left(\epsA_a- 1/\savg,\epsV_v- 1/\tavg ,\epsV_{a'},\epsA_{v'}\right)}
		\left(1 +O\left(\frac{1}{ \mu  n\ell }\right)\right)\right)^{-1}. 
\end{align}
By expanding in powers of $\epsA_b$ and $\epsV_{b'}$ 
we obtain 
\begin{align*}
\rho(\epsA_b,\epsA_a,\epsV_{b'},\epsV_{a'}) 
\cdot \frac{1-\pi(\epsA_b- 1/\savg ,\epsV_v- 1/\tavg ,\epsV_{b'},\epsA_{v'})}
			{1-\pi(\epsA_a- 1/\savg ,\epsV_v- 1/\tavg ,\epsV_{a'},\epsA_{v'})} 
			& = K + O(\eps^4) 
\end{align*} 			
where $K$ is a polynomial in $\epsA_b$ and $\epsV_{b'}$ of total degree at most 3 in $\{\epsA_b,\epsV_{b'}\}$.  
Calculations using the  size variables $y_1$ and $y_2$  as above show  that $K = k_{00} + k_{10} \epsA_b + k_{01} \epsV_{b'}  +  k_{20} {\epsA_b}^2 + O(\eps^4)$  for some $k_{ij}$ 
(and in particular, the  coefficients $k_{11}$ and $k_{02}$, and those for terms of third order, are all absorbed by the error terms). 
We also note from the definition of $\pi$ and $\rho$ that $k_{01}$ has $\delta^{\mathrm{di}}$ as a factor. 
Also, for $v\in T$ we have $\cA(v) = S$ in the bipartite case and $\cA(v) = S\sm\{\mate{v}\}$ in the digraph cases. Then the main summation over~$b$  in~\eqn{aux1552}  can be evaluated as 
\begin{align*}
\ell\cdot k_{00} +  \ell \sigmaA^2 k_{20}/\savg^2
	+\delta^{\di} \left( - k_{00} - k_{10}\epsA_{v'} -  k_{01}\epsV_v - k_{20} {\epsA_{v'}}^2\right),  
\end{align*}
with relative error $O(\eps^4)$, noting  
that $K$ has constant order, 
where we use that $\sum_{b\in S} \epsA_b = 0$ and that $\sum_{b\in S} {\epsA_b}^2 = \ell \sigmaA^2/{\savg}^2$.  
Using the size variables $y_1$ and $y_2$ as described above, we then find  that  $\Pc(P^*,R^*)_{av}(\dv) = \pi (1+O(\mu\eps^4))$, with the extra factor $\mu$ arising in the error term in the same way as for $\Rc$ in part (a). Part (b) follows. 

Part (c) is more straightforward than the first two parts and is easily verified by similar considerations, so we omit details.
\end{proof}

\begin{proof}[Proof of \thref{t:mainbip}]
Let $\mu_0$, $\degspread$, $\ell$, $n$, $m$, $\D$ be as in the theorem statement. Note that all sequences $\dv$ in $\D$ have the same values 
$\mu = \mu(\dv)=m/n(\ell-\delta^{\di})$, $\savg = \savg(\dv) = m/\ell$ and $\tavg = \tavg(\dv) = m/n$. All other sequences $\dv'$ in this proof will be close enough to $\D$ (in Hamming distance)  that $\mu':=\mu(\dv') \sim \mu$. 
Let $\eps= \max\{s^{\degspread-1},t^{\degspread-1}\}$ and note that 
$\eps\geq \max\{1/\sqrt{\savg}, 1/\sqrt{\tavg}\}$ by the lower bound on $\degspread$. 
	
We follow the template from Subsection~\ref{s:template},   first considering the ratios $R_{ab}$ for $a,b\in S$. 
Recall that $P$, $Y$ and $R$ denote the {\em actual} edge probability, path probability and ratio functions (c.f.~\eqref{realRatio} and \eqref{realProb}) and that $P^*$, $Y^*$ and $R^*$ are functions of degree sequences with closed form given at the beginning of this section. 
Recall that $P$, $Y$ and $R$ are always defined with respect to some underlying set $\cA$ which is either $\cA^{\bi}$ or $\cA^{\di}$ here. 
The next claim states that the functions $P^*$ and $R^*$ approximate $P$ and $R$ sufficiently well.  An analogous statement  for $Y^*$ and $Y$ appears in the proof of the claim but is not needed elsewhere.  
It is easy to see inductively that we only require $R_{ab}(\dv)$ when $a$ and $b$ are in the same set of the bipartition. 
Let $Q_1^{S}$ denote the set of sequences $\dv=(\sv,\tv)\in \Z^{\ell+n}$ such that $\dv-\ea\in \D$ for some $a\in S$; and let $Q_1^{T}$ denote the set of sequences $\dv=(\sv,\tv)\in \Z^{\ell+n}$ such that $\dv-\ev\in \D$ for some $v\in T$. 

\begin{claim}\thlab{RDenseBip}
Let $*$ be either $\bi$ or $\di$. 
For $\dv\in \D$, $av\in\cA$, 
\bee\lab{appP}
P_{av}(\dv) = P_{av}^*(\dv) (1+O(\mu\eps^4)),
\ee
and uniformly for all $\dv\in Q_1^S$ and all $a,b\in S$ 
\bee\lab{appR}
R_{ab}(\dv) = R_{ab}^*(\dv) (1+O(\mu\eps^4)).
\ee
\end{claim}
By symmetry, \eqref{appR} also holds for all $\dv\in Q_1^T$ and all distinct $a,b\in T$. 
 The proof is very similar to the proof of Claim~6.4 in~\cite{lw2018} with some adaptations to the bipartite setting.   We include a full proof for the sake of completeness. 

\begin{proof}[Proof of the claim.]
To show that $P$ and $P^*$ (and $R$ and $R^*$) are $(\mu\eps^4)$-close  in the sense of~\eqref{appP} and~\eqref{appR}, we define the compositional operator  
$$\Cc(\pv,\yv)= \big(\hat\pv, \Yc(\hat\pv, \yv)\big),\ \mbox{where\, $\hat\pv = \Pc(\pv,\Rc(\pv,\yv))$}.$$
  We first observe that $\cC$ fixes $(P,Y)$,  where in this context  we regard $P$ to be the function $\pv$ with $\pv_{ av}=P_{av} $ for all $av\in \cA$, and similarly $Y$ to be $\yv$ with $\yv_{avb}= Y_{avb}$, by \thref{l:recurse}. We will deduce a certain contraction property of $\cC$ by applying \thref{l:errorImplication}(a--c) one after the other, and then 
show that for any integer $k>0$, $\cC^{k}(\Pgr,\Ygr)$ and $\cC^{k}(P,Y)$ are $O(\mu)^{k }$-close.   We will also show that $\Pgr$ and $\cC^{k}(\Pgr)$ are $O(\mu\eps^4)$-close. These observations will then be shown to imply \thref{RDenseBip}.

Fix $k_0 = 4\log n$ and $r=4k_0+4=O(\log n)$.   
Let $\Omega^{(0)}$ be the set of sequences $\dv' \in \Z^n$ that are at $L_1$ distance at most $r$ from a sequence in $\D$.  Let $ \mu_1=5\mu$, and  define $\Omega^{(s)}$ to be the set of sequences $\dv'\in\Omega^{(0)}$ of $L_1$ distance at least $s+1$ from all sequences outside   $\Omega^{(0)}$. 

Towards applying \thref{l:errorImplication} we first establish that $(P,Y)$ and $(\Pgr,\Ygr)$ are elements of $\Pi_{ \mu_1}(\Omega^{(2)})$ (see \thref{Pi-defn}).
 Note that for $\dv'\in \Omega^{(0)}$, the values of  $\savg(\dv')$, $\tavg(\dv')$  and $\mu(\dv')$ are asymptotically equal to  $\savg$, $\tavg$  and $\mu$, respectively, since $M_1(\dv')=M_1(\dv_0)+O( \log n)$ for some sequence $\dv_0\in \D$. Thus, $\mu$ and  $\mu(\dv')$ are interchangeable in the error terms below. Furthermore, we note that the bounds on $s_a$ and $t_v$ in the theorem statement imply that for all balanced $\dv \in \Omega^{(0)}$, $s_a\sim s$ and $t_v\sim t$ uniformly for all $a\in S$ and $v\in T$. By \thref{lem:bipRealisable}\eqref{lem:bipRealisablei} we obtain  $\cN(\dv)>0$ for all balanced $\dv\in \Omega^{(0)}$ in both cases  $\cA^{\bi}$ and $\cA^{\di}$. In doing so, we may take $C=1$ and $F$ to be either empty in the bipartite case, or a matching in the digraph case; the condition $m\le \ell n/9$ 
 follows from choosing $\mu_0$ small enough, and the conditions on $\Delta_S$ and $\Delta_T$  can be seen to follow from $|s_a-\savg|\le \savg^{\degspread}$ for all $a\in S$, 
 $|t_v-\tavg|\le \tavg^{\degspread}$ for all $v\in T$ and $\degspread <1$, say.
After this, for $n$  and $\ell$  sufficiently large,~\thref{l:simpleSwitching}, together with 
the facts that  $s_a\sim s$ and $t_v\sim t$ uniformly for all $a\in S$, $v\in T$,  implies that for all $av\in \cA$ 
\bel{Pbound2}
P_{av}(\dv)\leq\frac{\mu}{1- 2\mu}(1+o(1)) <  \frac{5\mu}{4} 
\quad  \mbox{for all balanced $\dv\in \Omega^{(0)}$},
\ee 
where for the last inequality we use that $\mu< \mu_0  <  1/11$, say.  
Since $\Omega^{(2)}\se \Omega^{(0)}$ and $5\mu/4 <\mu_1$ this establishes requirement~\ref{Pi-a} for $P$ in the definition of $\Pi_{\mu_1}(\Omega^{(2)})$.
Now restrict slightly to $\dv \in \Omega^{(2)}$. 
 By definition $Y_{avb}(\dv)$ is the probability that both edges $av$ and $bv$ are present.
Hence  \thref{l:recurse}(c) implies (with the above bounds on $P_{av}(\dv)$ applying for all $\dv\in \Omega^{(0)}$) that $0\le Y_{avb}(\dv) =Y_{bva}(\dv) \le  3\mu P_{bv}(\dv)/2$ 
  (easily) assuming, as we may, that $\mu$ is sufficiently small. Thus $(P,Y)$ 
 satisfies condition~\ref{Pi-c} for membership of $\Pi_{\mu_1}(\Omega^{(2)})$, and also
 $$
 \sum_{v\in  \cA(a)\cap \cA(b)} Y_{avb}(\dv) \le  \sum_{v\in  \cA(a)\cap \cA(b)}\frac{3P_{bv}(\dv)\mu}{2}\le 3 \mu s_a(1+o(1)) 
 $$
  for all distinct $a,b\in S$, 
using $P_{bv}(\dv)\le 2\mu\sim 2s_a/n$
 which follows from \eqref{Pbound2} and recalling that $s_a\sim \mu n$ uniformly for all $a\in S$ for all $\dv \in \Omega^{(0)}$.  As  $4\mu < \mu_1$, this shows $Y$ satisfies condition~\ref{Pi-b} for membership of $\Pi_{\mu_1}(\Omega^{(2)})$ when $n$ is sufficiently large  and $a,b\in S$ are distinct. The equivalent statement for both $a,b\in T$ follows analogously. Note that this covers all cases of~\ref{Pi-b} as otherwise $\cA(a)\cap \cA(b)=\emptyset$ and the statement is trivial.
To see that $(\Pgr,\Ygr)$ is also in $\Pi_{ \mu_1}(\Omega^{(2)})$,   
we first observe that $P^*_{av}(\dv)\sim \mu$ for all $av\in \cA$ and all $\dv\in \Omega^{(0)}$. This is because $\eps\to 0$ since $\degspread <1$ and both $s,t\to\infty$, and because $\sigmaB^2/t\ell, \sigmaA^2/sn=O(\eps^2\mu)$.  
Properties~\ref{Pi-a}-\ref{Pi-c} follow directly from this fact and the definition of $Y^*$  since $\mu=\mu_1/5$.

Now for large $n$  and $\ell$  and for $(a,v)\in \oA$ we have $P_{av}(\dv)=P^*_{av}(\dv)(1\pm1)$ for all balanced $\dv\in \Omega^{(0)}$ since $P^*_{av}(\dv)\sim \mu$ and by~\eqref{Pbound2}.  Also,  $0\le Y_{avb}(\dv)\le 3\mu P_{bv}(\dv)/2$ implies  $Y_{avb}(\dv)=Y^*_{avb}(\dv)(1\pm1)$ for balanced $\dv\in \Omega^{(2)}$.  
We may now apply \thref{l:errorImplication}(a) with $\xi=1$ for any $S$-heavy $\dv \in \Omega^{(3)}$  to deduce that
$$
\Rc(P ,Y ) _{ab}(\dv)= \Rc(\Pgr ,\Ygr )_{ab}(\dv)(1+O (\mu_1))
$$
for all  $a,b\in S$ (noting that for $a=b$ the claim is trivial, and for distinct $a,b$ we use the previous observations).  
Writing $\rv$ for $\Rc(P ,Y )$ and $\rv'$ for $\Rc(\Pgr ,\Ygr )$ we obtain from this 
and \thref{l:errorImplication}(b) that 
$$
\Pc(P ,\rv ) _{av}(\dv)= \Pc(\Pgr ,\rv' )_{av}(\dv)(1+O (\mu_1  ))
$$ 
for all balanced   $\dv \in \Omega^{(4)}$  and all $av\in \cA$.   
Next applying \thref{l:errorImplication}(c) in the same way to balanced $\dv \in \Omega^{(6)}$ gives 
$$
\Yc(\hat\pv ,Y) _{avb}(\dv)= \Yc(\hat\pv' , \Ygr )_{avb}(\dv)(1+O (\mu_1  ))
$$ 
for all balanced $\dv \in \Omega^{(6)}$, and all  $(a,v,b)\in \oA_2$ with $a\in S$,   
where $\hat\pv = \Pc(P ,\rv )$ and $\hat\pv' = \Pc(\Pgr ,\rv' )$. 
 Recalling the definition of $\Cc(\pv,\yv)$ from the beginning of this proof, we note that 
  the  three conclusions above imply that, for $\cC(P ,Y )=(P_1,Y_1)$ and $\cC(\Pgr ,\Ygr )=(P_1^*,Y_1^*)$, we have $P_1^*(\dv)=P_1(\dv)(1+O(\mu_1))$ and  $Y_1^*(\dv)=Y_1(\dv)(1+O(\mu_1))$ for all $\dv \in \Omega^{(6)}$. Similarly, making $k-1$ iterated applications of the three parts of \thref{l:errorImplication} with ever-decreasing $\xi$ produces
$$
P_k^*(\dv)=P_k(\dv)(1+O(\mu_1)^k) \mbox{and} Y_k^*(\dv)=Y_k(\dv)(1+O(\mu_1)^k)
$$
for all $\dv \in \Omega^{(2k+2)}$,  where $P_k$, $P_k^*$ etc.\ are defined analogously for $\cC^{k}$. 
In the same way, applying  \thref{l:mapleHigher}(a), (b) and (c) in turn,   recalling Lemma~\ref{l:errorImplication} to handle the small error terms, shows that  $P_1^*(\dv)=P^*(\dv)(1+O(\mu \eps^4))$ and  $Y_1^*(\dv)=Y^*(\dv)(1+O(\mu \eps^4))$ for all $\dv \in \Omega^{(6)}$.
 Using the three parts of \thref{l:errorImplication} repeatedly, and bounding the total distance moved during the iterations as for a contraction mapping  (as the sum of a geometric series), this gives
$$
P_k^*(\dv)=P^*(\dv)(1+O(\mu \eps^4))   \mbox{ and  }  Y_k^*(\dv)=Y^*(\dv)(1+O(\mu \eps^4))
$$
for all $\dv \in \Omega^{(2k+2)}$. Using the last two conclusions with $k:= k_0 = 4 \log n$ and the fact that $P_k = P$ (as $\Cc$ fixes $(P,Y)$) gives~\eqn{appP} for all balanced $\dv\in \Omega^{(r-2)}$ since
   we may assume that the  $ O( \mu_1)$ term is at most $1/e$ say. (Recall that $\mu_1=5\mu < 5\mu_0$ which we may choose to be sufficiently small.)
Note that $\D\se  \Omega^{(r)}\subseteq \Omega^{(r-2)}$ by definition.  
 For~\eqn{appR}, we now use that~\eqn{appP} holds for all balanced $\dv\in \Omega^{(r-2)}$ to deduce from \thref{l:errorImplication}(a) that 
$\Rc(P,Y)_{ab}(\dv) = \Rc(\Pgr ,\Ygr )_{ab}(\dv) (1+O(\mu \eps^4))$ for all $S$-heavy $\dv\in \Omega^{(r-1)}$. 
This, together with (a) above and the fact that $\Rc(P,Y)=R$,   
implies~\eqn{appR} for all $\dv\in Q_1^S \subseteq   \Omega^{(r-1)}$.
\end{proof} 

Moving on to Step 2 of the template, we now make suitable definitions of probability spaces $\cS$ and $\cS'$ in preparation for using \thref{l:lemmaX}. Let $\Omega$ be the underlying set of 
$\cB_m(\ell,n)$ in the bipartite case and of $\vec \cB_m(n)$ in the digraph case, that is, the set of all degree sequences $\dv=(\sv,\tv)\in \Z^{\ell+n}$ such that $M_1(\sv)=M_1(\tv)=m$ (and $\ell =n$ in the digraph case). Let $\cS'=\cD(\G(\ell,n,m))$ in the bipartite case, and $\cS'=\cD(\vec \G(n,m))$ in the digraph case. 
Set 
$$
\W=\{\dv=(\sv,\tv)\in \D : 
 \sigma^2(\sv) \le  2s,\ \sigma^2(\tv) \le 2t,\left|\delta^{\di}\sigma(\sv,\tv)\right|<\xi s  \}
$$
where $\xi= \max\{ \log^2\ell/\sqrt{\ell},\log^2n/\sqrt{n}\}$.

Define the graph $G$ on vertex set $\D$ by joining two degree sequences by an edge if they are of the form $\dv-\ea$ and $\dv-\eb$ for some $\dv\in Q_1^S$ and $a,b\in S$, or for some $\dv\in Q_1^T$ and $a,b\in T$. 
We note at this point that the diameter of $G$ is $r=O(\ell s^{\degspread}+nt^{\degspread})=O(\eps\mu n\ell)$ since the constant sequence $(d,\ldots,d)$ is an element of $\D$ and by the degree constraints for $\dv\in \D$. The same bound holds for the diameter of $G[\W]$. 
By definition of $R$ in~\eqref{realRatio} and its approximation in~\eqref{appR}  we   have,   for adjacent vertices/sequences in this graph, that 
\begin{align}\lab{targetLHSnew}
\frac{\Pr_{\cS'}(\dv-\ea)}{\Pr_{\cS'}(\dv-\eb)}
&= R_{ab}(\dv) = 
R^*_{ab}(\dv) \left(1+ O\left(\mu \eps^4\right)\right). 
\end{align}

We now define the ideal probability space $\cS$ (see Step 3 of the template) on $\Omega$. Recall the definition of $\correct(\dv)$ in~\eqref{corrH}, which is slightly different in the bipartite case and the digraph case (due to $\mu$ being defined slightly differently and the extra term in the exponent in the digraph case). Also recall  (from just before~\eqref{H}) that we use $\apx(\dv)$ to denote the product $\Pr_{\cB_m}(\dv)\correct(\dv)$ on the right hand side of~\eqref{enumFormula}, where $\cB_m = \cB_m(\ell,n)$ in the bipartite case and $\cB_m=\vec\cB_m(n)$ in the digraph case. 
Then set 
$$
\Pr_{\cS}(\dv) = \apx(\dv)/\sum_{\dv'\in\W}\apx(\dv')
	= \frac{\apx(\dv)}{\ex_{\cB_m}(\mathbbm{1}_{\W}\correct)} 
$$
for $\dv\in \W$ and $\Pr_{\cS}(\dv)=0$ otherwise. 

Following the template to Step 4 we next estimate the probability of $\W$ in the two probability spaces $\cS$ and $\cS'$ in both cases, bipartite and digraph. We simultaneously evaluate $\Pr_{\cB_m}(\W)$ and  $\ex_{\cB_m}(\mathbbm{1}_{\W}\correct)$  for later use.  First note that $\Pr_{\cS}(\W)=1$ by definition. Next, $M_1(\sv)=M_1(\tv)=m$ for all $\dv=(\sv,\tv)\in\Omega$, by definition. Let $\bar n =\min\{n,\ell\}.$ For   the following, let  $\dv$  be chosen  in either of $\cD(\G)=\cS'$ and $\cB_m$, in either of  the bipartite and the digraph case. Then $|s_a-\savg|\le \savg^{\degspread}$ and $|t_v-\tavg|\le \tavg^{\degspread}$ for all $a\in S$ and $v\in T$ with probability at least $1-O({\bar n}^{-\omega})$
by \thref{l:sigmaConcBip}(i) and since  $s> (\log \ell)^K$ and $t> (\log n)^K$  for all $K>0$. It follows that $\Pr_{\cD(\G)}(\D) = \Pr_{\cS'}(\D) = 1-O(\bar n^{-\omega})$ and 
$\Pr_{\cB_m}(\D) = 1-O(\bar n^{-\omega})$. 
  Next, apply  \thref{l:sigmaConcBip}(ii) with $\alpha =\xi/2$ for $\dv=(\sv,\tv)$ (noting that $m\ge s+t\gg\log^3 \ell+ \log^3 n$ and $(\log^3 n+\log^3 \ell)/m=o(\xi^2)$ by definition of $\xi$)  to deduce that 
$\sigma^2(\sv)=s(1-\mu)(1\pm \xi)$, $\sigma^2(\tv)=t(1-\mu)(1\pm \xi)$ and $\sigma(\sv,\tv)=O(\xi s)$ with probability $1-O(\bar n^{-\omega})$. 
This implies in particular that $\Pr (\W) = 1-O(\bar n^{-\omega})$ in both $\cS'$ and $\cB_m$. 
Thus $\Pr(\W)\ge 1-\eps_0$ in both $\cS$ and $\cS'$ in both cases, bipartite and digraph, for, say, $\eps_0=1/\bar n$.

 Before moving on to Step 5 of the template we use these concentration results to estimate $\ex_{\cB_m}(\mathbbm{1}_{\W}\correct)$. If $\sigma^2(\tv)=t(1-\mu)(1+O(\xi))$ then the term $\sigma^2(\tv)/t(1-\mu)$ in the exponent of $\correct(\dv)$ is $1+O(\xi)$. Similarly for the term $\sigma^2(\sv)/s(1-\mu)$. Furthermore, in the digraph case, if $\sigma(\sv,\tv)=O(\xi s)$ then the term $\delta^{\di}\sigma(\sv,\tv)/s(1-\mu)$ in $\correct(\dv)$ is $O(\xi)$. It follows, using the strong concentration results in the previous paragraph, that 
\bel{HConcDense}
\correct(\dv)= 1+O(\xi) \text{ with probability } 1-O(\bar n^{-\omega}) \text{ for } \dv\in\cB_m.
\ee
Furthermore, by definition of $\W$ it follows that $\correct(\dv) = \Theta(1)$ for all $\dv\in \W$, using the fact that $\mu < 1/2,$ say. This and~\eqref{HConcDense} then imply that $\ex_{\cB_m}(\mathbbm{1}_{\W}\correct)=1+O(\xi+\bar n^{-\omega})=1+O(\xi).$

We now move to Step 5 in the template. 
For  $\dv=(\sv,\tv)\in Q_1^S$ and $a,b\in S$, 
\begin{align}\label{eq:ratioDense}
\frac{\apx(\dv-\ea)}{\apx(\dv-\eb)} &=
\frac{\apx(\sv-\ea,\tv)}{\apx(\sv-\eb,\tv)}
=  \frac{s_a(n+\delta^{\rm bi} -s_b)}{s_b(n+\delta^{\rm bi}-s_a)}
 \exp\bigg(\frac{s_b-s_a}{s(1-\mu)\ell }
 \left(1-\frac{\sigma^2(\tv)}{t(1-\mu)}\right)
 +\frac{ \delta^{\rm di} (t_{a'}-t_{b'})}{t(1-\mu)n}\bigg),  
\end{align}
by~\eqref{H}. 
Compare the expression on the right hand side with $R^*=\rho$ at the beginning of this section (after a straight-forward reparameterisation) to see that 
\bel{targetRHS}
 \frac{\apx(\dv-\ea)}{\apx(\dv-\eb)} 
= R^*_{ab}(\dv)  \left(1+ O\left( 1/\bar n^2\right)\right),  
\ee
where the error is due to the fact that we use $e^x = 1+x+O(x^2)$ and that $\mu(\dv) = \mu +O(1/n\ell)$ for $\dv\in Q_1^S$.   
This together with
\eqref{targetLHSnew} gives
\bel{ratioInD}
\frac{\Pr_{\cS'}(\dv-\ea)}{\Pr_{\cS'}(\dv-\eb)} =e^{O(\mu\eps^4)} \frac{\apx(\dv-\ea)}{\apx(\dv-\eb)} 
\ee
for $\dv\in Q_1^S$ and $a,b\in S$. 
The same applies for $\dv\in Q_1^T$ and $a,b\in T$ by symmetry. 
 Note that when both $\dv-\ea$ and $\dv-\eb$ are elements of $\W$ then the right hand side is in fact equal to 
$ e^{O(\mu\eps^4)}\Pr_{\cS}(\dv-\ea)/\Pr_{\cS}(\dv-\eb),$ by definition of $\Pr_{\cS}$ above. 
Therefore, by \thref{l:lemmaX} 
\begin{align*}
\Pr_{\cS'}(\dv) &= \Pr_{\cS}(\dv) e^{O\left( r\mu\eps^4+\eps_0 \right)}\\
&= \Pr_{\cB_m}(\dv)\correct(\dv)\left(1+O\left(\xi + r\mu\eps^4+\eps_0\right)\right) 
\end{align*} 
for all $\dv\in \W$,  where we use that $\ex_{\cB_m} (\mathbbm{1}_{\W}\correct) = 1+O(\xi)$. 
Now let $\dv\in \D\sm\W$. Then there is some sequence $\dv'\in \W$ such that the distance of $\dv$ and $\dv'$ in $G$ is at most $2r$. Any two adjacent sequences $\tilde\dv-\ea$ and $\tilde\dv-\eb$ along that path satisfy~\eqref{ratioInD}, so that by telescoping we see that 
\begin{align*}
\Pr_{\cS'}(\dv) &= H(\dv) \left(1+O\left(\xi + r\mu\eps^4+\eps_0\right)\right)
\end{align*} 
for such $\dv$ as well. This proves the claim, 
since $\xi = \max\{ \log^2\ell/\sqrt{\ell},\log^2n/\sqrt{n}\}$, $\eps_0=1/n$, and 
$r\mu\eps^4 = O\left(\eps^5\mu^2n\ell\right)$.
\end{proof}

\begin{proof}[Proof of \thref{t:bipmodel}]
In the sparse case, given that the set $\D$ in \thref{t:sparseCaseBip} is nonempty, we define a set $\W$ in the proof of \thref{t:sparseCaseBip} and show  that $\Pr_{\cS'}(\W) =  1-o(n^{-\omega})$ just before~\eqn{HConc}, where $\cS'= \cD(\vec \G(n,m))$  or $\cD(\G(\ell,n,m))$ as the case may be, and also that $\Pr_{\cB_m}(\W) =  1-o(n^{-\omega})$. 
As observed in that proof, the formula given in \thref{t:sparseCaseBip} holds for all $\dv\in\W$, and also we can assume that $\corr{\dv}\sim 1$ by~\eqn{HConc}. This gives the required a.q.e.\ property, for any triple $(\ell,n,m)$ such that the set $\D$ in that theorem is nonempty (for $n$ sufficiently large).
The proof of \thref{t:mainbip} implies a similar result, for any  $(\ell,n,m)$ satisfying that theorem's hypotheses. So all that is left to do is check which ordered triples $(\ell, n, m)$ are covered by these results,   and to supply the remaining cases using previously known results. 

We first concentrate on the bipartite case~\ref{model-b}. We claim that the conditions in \ref{model-b-ii} imply that the hypotheses for \thref{t:mainbip} are satisfied. Note first that $n^3 = o(\ell^2m^{1-\eps})$ implies in particular that $\ell\to\infty$ with $n$, using the trivial $m\le n\ell$. The same asymptotic inequality, together with $\ell\le n$, implies that $n\ell^{\eps}\le m$. The required bounds on $m$ in \thref{t:mainbip} then follow from $n\ell^{\eps}\le m<\mu_0 n\ell$ by choosing, say, $\degspread = 1/2+\eps/10.$
Next, we check that the conditions in \ref{model-b-iii} imply that the hypotheses for \thref{t:sparseCaseBip} are satisfied and that there exists some non-empty set $\D$ as in the theorem statement. Given $\ell,n$ and $m$, let $\dv=(\sv,\tv)$ be a sequence such that $M_1(\sv)=M_1(\tv)=m$, all elements of $\sv$ are equal to $\lfloor m/\ell\rfloor$ or $\lceil m/\ell\rceil$ and all elements of $\tv$ are equal to $\lfloor m/n\rfloor$ or $\lceil m/n\rceil$, respectively. 
Then the inequality 
$$ 
\Delta_S^3\Delta_T^3(n\ell)^{\eps/2} 
\le \left(\frac{m}{\ell}+1\right)^3 \left(\frac{m}{n}+1\right)^3m^{\eps} \ll \min\{sm,tm\}
$$ 
can be seen to follow from $m^{4+\eps}=o(n^2\ell^2\min\{\ell,n\}).$  
Thus, for $n$ sufficiently large, the set $\D$ in \thref{t:sparseCaseBip} can be assumed to be non-empty and for \ref{model-b-iii}, the requirements of \thref{t:sparseCaseBip} are satisfied.
For $\log^3\ell+ \log ^3n\ll m \ll \min\{n/\log^2n,\ell/\log^2\ell\}$, \thref{l:sigmaConcBip} shows that with probability $1-n^{-\omega}-\ell^{-\omega}$ all vertices in $\G(\ell,n,m)$ have bounded degrees, so the main theorem of  Bender~\cite{b1974} applies. \thref{l:sigmaConcBip}  also shows that the terms $\sigma^2(\sv)$ and $\sigma^2(\tv)$ are  concentrated near their  expected values, to the extent that $\corr{\dv} =1+o(1)$ with probability $1-n^{-\omega}-\ell^{-\omega}$. It is then a simple calculation to check that Bender's  formula corresponds asymptotically with the first formula in~\eqn{probbin}, noting  the adjustment required to count graphs (as in Remark~\ref{r:first}). This finishes the proof of part (b).

For the digraph case~\ref{model-a}, \thref{t:mainbip,t:sparseCaseBip} also supply the required statement when $\ell=n$. Here \thref{t:mainbip} covers the range $n^{1+\eps}<m<\mu_0 n^2$ for any $\eps>0$ and sufficiently small constant $\mu_0$, and \thref{t:sparseCaseBip} covers the range $n/\log^3n < m < n^{5/4-\eps},$ say. Only the very dense and very sparse remain. McKay and Skerman~\cite[Theorem 1(d)]{MS}  immediately covers $\min\{m, n(n-1)-m) > n^2/c\log n$  for any  $c>0$. Larger values of $m$ than this are covered by complementation of the results for smaller values. On the other hand, for $\log ^3n\ll m= \ll n/\log^2n$, we  may again use  \thref{l:sigmaConcBip} and the main theorem of~\cite{b1974}. This works almost exactly the same as for the bipartite graph case. 
\end{proof}

\begin{proof}[Proof of \thref{t:edgeprobability}]
The stated formula is a by-product of the proof of \thref{t:mainbip}, so here we just point to the relevant spots within that proof. 
Recall the definition of $P^\ast$, which is our approximation to the edge probability, at the beginning of Section~\ref{s:denseBip}, and note that a parameterisation yields the formula given in the statement of \thref{t:edgeprobability}. 
\thref{RDenseBip} then yields the desired approximation since $\mu\eps^4=O(\min\{s,t\}^{4\degspread-4}m/n\ell)$.
\end{proof}

 \end{document}